\documentclass[11pt]{amsart}
\usepackage[utf8]{inputenc}
\usepackage{amssymb,amsmath}
\usepackage{graphicx}
\usepackage{color}
\usepackage{tikz}
\usepackage{pgfplots}
\usepackage{enumitem}
\usepackage{algorithm}
\usepackage{setspace}
\usepackage{algpseudocode}
\usepackage{mathtools}
\usepackage{xfrac}

\usepackage{accents}
\usepackage{multicol} 
\usepackage{multirow}
\usepackage{soul}
\usepackage{subcaption}
\usepackage{booktabs}
\usepackage{array}
\usepackage{comment}
\newcolumntype{L}[1]{>{\raggedright\let\newline\\\arraybackslash\hspace{0pt}}m{#1}}
\newcolumntype{C}[1]{>{\centering\let\newline\\\arraybackslash\hspace{0pt}}m{#1}}
\newcolumntype{R}[1]{>{\raggedleft\let\newline\\\arraybackslash\hspace{0pt}}m{#1}}
\usepackage{siunitx}
\usepackage{bbold}

\usepackage[hidelinks]{hyperref}
\usepackage[nameinlink,noabbrev]{cleveref}
\newtheorem{theorem}{Theorem}
\newtheorem{proposition}[theorem]{Proposition}

\theoremstyle{definition}

\theoremstyle{lemma}

\theoremstyle{remark}
\newtheorem{remark}[theorem]{Remark}

\newtheorem{assumption}[theorem]{Assumption}

\Crefname{assumption}{Assumption}{Assumptions}
\numberwithin{theorem}{section}
\numberwithin{equation}{section}
\numberwithin{table}{section}
\numberwithin{figure}{section}
\definecolor{myBlue}{RGB}{30,144,255} 
\definecolor{myGreen}{RGB}{69,169,0} 
\definecolor{myRed}{RGB}{165,12,42} 
\definecolor{myOrange}{RGB}{225,92,22} 
\definecolor{mycolor0}{rgb}{0.12156862745098,0.466666666666667,0.705882352941177} 
\definecolor{mycolor1}{rgb}{0.00000,0.44700,0.74100}
\definecolor{mycolor2}{rgb}{0.85000,0.32500,0.09800}
\definecolor{mycolor3}{rgb}{0.49400,0.18400,0.55600}
\definecolor{mycolor4}{rgb}{0.92900,0.69400,0.12500}
\definecolor{mycolor5}{rgb}{0.46600,0.67400,0.18800}
\definecolor{mycolor6}{rgb}{0.30100,0.74500,0.93300}
\definecolor{mycolor7}{rgb}{0.63500,0.07800,0.18400}

\newcommand{\delete}[1]{ }

\def\N{\mathbb{N}}

\def\R{\mathbb{R}}

\newcommand\dx{\,\mathrm{d}x}

\newcommand{\hook}{\ensuremath{\hookrightarrow}}

\newcommand{\calO}{\ensuremath{\mathcal{O}} }
\newcommand{\calP}{\ensuremath{\mathcal{P}} }
\newcommand{\Q}{\ensuremath{\mathcal{Q}} }

\newcommand{\calT}{\ensuremath{\mathcal{T}} }
\newcommand{\V}{\ensuremath{\mathcal{V}}}

\newcommand{\cHV}{\ensuremath{{\mathcal{H}_{\V}}} }
\newcommand{\cHQ}{\ensuremath{{\mathcal{H}_{\Q}}} }

\newcommand{\f}{\ensuremath{f}}
\newcommand{\g}{\ensuremath{g}}
\newcommand{\Amat}{\mathbf{A}}
\newcommand{\Bmat}{\mathbf{B}}
\newcommand{\Cmat}{\mathbf{C}}
\newcommand{\Dmat}{\mathbf{D}}
\newcommand{\Imat}{\mathbf{I}}
\newcommand{\Mmat}{\mathbf{M}}
\newcommand{\Nmat}{\mathbf{N}}
\newcommand{\Smat}{\mathbf{S}}
\newcommand{\Tmat}{\mathbf{T}}
\newcommand{\Cfg}{\ensuremath{C_\text{rhs}}}
\newcommand{\Ctau}{\ensuremath{\Cmat_\tau}}

\DeclareFontFamily{U}{matha}{\hyphenchar\font45}
\DeclareFontShape{U}{matha}{m}{n}{
	<-6> matha5 <6-7> matha6 <7-8> matha7
	<8-9> matha8 <9-10> matha9
	<10-12> matha10 <12-> matha12
}{}
\DeclareSymbolFont{matha}{U}{matha}{m}{n}

\DeclareFontFamily{U}{mathx}{\hyphenchar\font45}
\DeclareFontShape{U}{mathx}{m}{n}{
	<-6> mathx5 <6-7> mathx6 <7-8> mathx7
	<8-9> mathx8 <9-10> mathx9
	<10-12> mathx10 <12-> mathx12
}{}
\DeclareSymbolFont{mathx}{U}{mathx}{m}{n}

\DeclareMathDelimiter{\vvvert} {0}{matha}{"7E}{mathx}{"17}%

\DeclareMathOperator{\newton}{N}
\DeclareMathOperator{\meter}{m}
\DeclareMathOperator{\millimeter}{mm}
\DeclareMathOperator{\second}{s}
\DeclareMathOperator{\minute}{min}
\DeclareMathOperator{\hour}{h}
\allowdisplaybreaks
\begin{document}
\title[A Novel Iterative Time Integration Scheme for Linear Poroelasticity]{A Novel Iterative Time Integration Scheme\\ for Linear Poroelasticity}
\author[]{R.~Altmann$^\dagger$, M.~Deiml$^\ddagger$}
\address{${}^{\dagger}$ Institute of Analysis and Numerics, Otto von Guericke University Magdeburg, Universit\"atsplatz 2, 39106 Magdeburg, Germany}
\address{${}^{\ddagger}$ Institute of Mathematics, University of Augsburg, Universit\"atsstr.~12a, 86159 Augsburg, Germany}
\email{robert.altmann@ovgu.de, matthias.deiml@uni-a.de}
%
%
\date{\today}
%
%
\begin{abstract}
Within this paper, we introduce and analyze a novel time stepping scheme for linear poroelasticity. In each time frame, we iteratively solve the flow and mechanics equations with an additional damping step for the pressure variable. Depending on the coupling strength of the two equations, we explicitly quantify the needed number of inner iteration steps to guarantee first-order convergence. Within a number of numerical experiments, we confirm the theoretical results and study the dependence of inner iteration steps in terms of the coupling strength. Moreover, we compare our method to the well-known fixed-stress scheme. 
\end{abstract}
%
%
\maketitle
%
%
%
%
\section{Introduction}
The equations of poroelasticity appear in various fields such as in geomechanics modeling porous rocks~\cite{DetC93,Zob10} or in medical applications modeling soft tissue~\cite{RNM+03, PieRV22}. The mathematical model, also called Biot's consolidation model, was introduced in~\cite{Bio41} and describes the dynamics of an elastic solid containing pores which are filled with a liquid. Hence, it considers a coupled system with one equation for the elastic deformation of the solid under physical stress and a second equation for the flow of the liquid arising from pressure differentials. 

In this paper, we focus on the time discretization of the spatially discretized poroelasticity model. For details on the spatial discretization we refer to~\cite{MurL92,PhiW07a,PhiW08,LeeMW17,HonK18} and the references therein. Related space--time approaches were considered in~\cite{BauRK17,ArfS22}. A direct application of the implicit Euler scheme results in an unconditionally stable first-order method~\cite{ErnM09}.
In each time step, however, it requires the solution of a large (coupled) linear system. Especially in three-dimensional applications, this may become a severe computational challenge. For this reason, recent work has focused on the search of more efficient methods. 
Iterative schemes, for instance, decouple the equations in the sense that either the flow or the mechanics equation is solved first, followed by the solution of the remaining problem~\cite{WheG07,MikW13}. This is referred to as {\em fixed strain}, {\em fixed stress}, {\em drained split}, or {\em undrained split}. All these methods have in common that the system is split into two smaller subsystems for which well-known preconditioners can be applied~\cite{LeeMW17}. Unfortunately, the needed number of inner iteration steps is unknown and the methods partially exhibit stability problems calling for a problem-specific tuning, cf.~\cite{KimTJ11a,KimTJ11b,StoBKNR19}. 

Yet another decoupling approach is the semi-explicit Euler method introduced in~\cite{AltMU21,AltM22}, see also the related splitting strategy~\cite{VabVK14}. Here, no inner iteration is needed but the scheme is only stable if the coupling of the mechanics and flow equation is sufficiently {\em weak}, which is expressed in terms of the material parameters. In contrast to the mentioned iterative schemes, however, it can also be generalized to construct higher-order schemes~\cite{AltMU22}. 

This paper is devoted to a novel time integration scheme which combines the iterative idea with the semi-explicit approach. More precisely, we consider an iterative scheme with an a priori specified number of inner iteration steps depending on the coupling strength of the elastic and the flow equation. This then yields a very efficient time stepping scheme for problems with a moderate coupling, which includes most geomechanical applications. The construction of the scheme is based on the semi-explicit Euler scheme extended with an inner fixed point iteration and a relaxation step. 
While the convergence of the scheme for an unlimited number of inner iteration steps can be shown easily, we prove an explicit upper bound for the number of iterations needed to guarantee first-order convergence. 

This paper is organized as follows. In Section~\ref{sec:prelim} we recall the equations of linear poroelasticity and introduce the parameter~$\omega$ which measures the coupling strength of the elastic and the flow equation. Afterwards, we summarize known time stepping schemes for the semi-discrete equations in Section~\ref{sec:methods}. Moreover, we introduce the novel iterative scheme including a damping step for the pressure variable. The proof of convergence is then subject of Section~\ref{sec:convergence}. Here, we present explicit bounds on the number of inner iteration steps in terms of the coupling parameter. Finally, we present three numerical examples in Section~\ref{sec:numerics} proving the competitiveness of the proposed scheme.

\subsection*{Notation}
As usual, $\| v\|$ denotes the Euclidean norm of a vector $v$. Moreover, we write $\|v\|^2_\Mmat = (\Mmat v,v) = \|\Mmat^{1/2} v\|^2$ for a symmetric and positive definite matrix~$\Mmat$. Here, $\Mmat^{1/2}$ is the square root of the matrix~$\Mmat$, which is again symmetric and positive definite. For vector-valued functions, we write~$\|v\|_{L^\infty(\Mmat)}$ for the maximum (over time) of $\|v(t)\|_{\Mmat}$. 
%
%
\section{Preliminaries}\label{sec:prelim}
In this preliminary section, we introduce the equations of poroelasticity, including the {\em coupling parameter} which plays a key role in this paper. Afterwards, we shortly discuss the spatial discretization of the system equations. 
%
%
\subsection{Linear poroelasticity and its weak formulation}
We consider a bounded Lipschitz domain~$\Omega \subseteq \R^m$, $m\in\{2,3\}$, as computational domain as well as a bounded time interval~$[0,T]$ with~$T>0$. The quasi-static Biot poroelasticity model~\cite{Bio41,DetC93,Sho00} reads as follows: find the deformation~$u\colon [0,T]\times\Omega\rightarrow\R^{m}$ and the pressure~$p\colon [0,T]\times\Omega\rightarrow\R$, satisfying 
%
\begin{align*}
	- \nabla\cdot\sigma(u) + \alpha \nabla p 
	&= {\hat f} \qquad\text{in } (0,T]\times\Omega, \\ 
	\partial_{t}\, \big(\alpha \nabla\cdot u + \tfrac{1}{M} p \big) + \nabla\cdot \big(\tfrac{\kappa}{\nu} \nabla p \big) 
	&= {\hat g} \qquad\text{in } (0,T]\times\Omega. 
\end{align*}
%
Therein, 
\[
	\sigma(u) 
	= 2\mu\, \varepsilon(u) + \lambda\, \big(\nabla \cdot u\big)\, \operatorname{id} 
	\qquad\text{with}\qquad
	\varepsilon(u) 
	= \tfrac12\, \big(\nabla u + (\nabla u)^*\big)
\]
denotes the stress tensor of the solid with the Lam{\'e} coefficients~$\lambda$ and~$\mu$. The remaining parameters are the Biot--Willis fluid--solid coupling coefficient~$\alpha$, the Biot modulus~$M$ describing how compressible the fluid is under pressure, the intrinsic permeability~$\kappa$, and the fluid viscosity~$\nu$. The right-hand sides model the external influence on the system. More precisely, ${\hat f}$ denotes the volumetric load and~${\hat g}$ the fluid source. 

For the well-posedness of the system, we further assume given initial data 
\[
	u(0,\bullet) = u^0, \qquad
	p(0,\bullet) = p^0
\]
satisfying the consistency condition
\[
	- \nabla\cdot\sigma(u^0) + \alpha \nabla p^0 
	= {\hat f}(0) \qquad\text{in } \Omega
\]
as well as boundary conditions. Within this paper, we restrict ourselves to homogeneous Dirichlet boundary conditions for~$u$ and for~$p$. Results on the unique solvability of the system are discussed in~\cite{Sho00}. 

For the weak formulation of the poroelastic equations, we introduce function spaces for the pressure and displacement fields -- including the homogeneous boundary conditions -- as well as respective $L^2$-spaces, namely 
\[
	\V \coloneqq \big[H^1_0(\Omega)\big]^m, \qquad
	\cHV \coloneqq \big[L^2(\Omega)\big]^m, \qquad
	\Q \coloneqq H^1_0(\Omega), \qquad 
	\cHQ \coloneqq L^2(\Omega).
\]
We pair $\V$ with the norm $\|\nabla \bullet\| \coloneqq \|\nabla \bullet\|_{[L^2(\Omega)]^{m,m}}$ for which the positive definiteness follows from the homogeneous boundary conditions.
We further introduce the bilinear forms~$a\colon\V\times\V\to\R$, $b\colon\Q\times\Q\to\R$, $c\colon\cHQ\times\cHQ\to\R$, and $d\colon\V\times\cHQ\to\R$ by 
\begin{align*}
	a(u,v) \coloneqq \int_\Omega \sigma(u) : \varepsilon(v) \dx, \qquad 
	&b(p,q) \coloneqq \int_\Omega \frac{\kappa}{\nu}\, \nabla p \cdot \nabla q \dx,\\
	c(p,q) \coloneqq \int_\Omega \frac{1}{M}\, p\, q \dx, \qquad 
	&d(u,q) \coloneqq \int_\Omega \alpha\, (\nabla \cdot u)\, q \dx 
\end{align*}
for $u,\, v \in \V$ and $p,\, q \in \Q$. Note that the bilinear forms $a$, $b$, and $c$ are symmetric and coercive with corresponding constants~$c_a$, $c_b$, and $c_c$, respectively. For $a$, which includes the classical double dot notation from continuum mechanics, this follows from Korn's inequality; see~\cite[Thm.~6.3-4.~(iv)]{Cia88}. Moreover, all four bilinear forms are bounded with stability constants denoted by~$C_a$, $C_b$, $C_c$, and $C_d$, respectively. In particular, we have $d(u, q) \le C_d\, \|u\|_\V \|q\|_{\cHQ}$ and, using integration by parts, also~$d(u, q) \le C_d\, \|u\|_\cHV \|q\|_{\Q}$.

Then, the weak formulation of linear poroelasticity seeks abstract functions~$u\colon [0,T] \to \V$ and $p\colon [0,T] \to \Q$ such that  
\begin{subequations}
	\label{eq:poro:weak}
	\begin{align}
		a(u, v) - d(v, p) 
		&= \langle{\hat \f}, v\rangle, \label{eq:poro:weak:a} \\
		d(\dot u, q) + c(\dot p, q) + b(p, q) 
		&= \langle{\hat \g}, q\rangle \label{eq:poro:weak:b} 
	\end{align}
\end{subequations}	
for all test functions $v \in \V$ and $q \in \Q$. Here, $\langle\bullet, \bullet\rangle$ denotes the duality pairing in~$\V$ and $\Q$, respectively. 
Note that the two equations decouple in the case~$\alpha=0$. More precisely, equation~\eqref{eq:poro:weak:a} would turn into an elliptic and~\eqref{eq:poro:weak:b} into a parabolic equation, which explains the expression of~\eqref{eq:poro:weak} being an {\em elliptic--parabolic equation}.
%
%
\subsection{Coupling parameter} \label{sec:prelim:omega}
For the simulation of poroelasticity, the coupling strength of the elliptic and parabolic equation plays an important role. For instance, certain time stepping schemes only converge for problems which are weakly coupled; see~\cite{AltMU21,AltMU22,KimTJ11b}. One possible measure reads 
\[
	\tilde \omega
	\coloneqq \frac{C_d^2}{c_ac_c},
\]
which is motivated, e.g., by the convergence analysis presented in~\cite{AltMU21}. Following~\cite[Sect.~6.3]{Cia88}, we have due to the assumed boundary conditions $2\, \|\varepsilon(v)\|^2 = \|\nabla v\|^2 + \|\nabla\cdot v\|^2_\cHQ$ and, hence, 
\begin{align*}
	a(v,v) 
	= \int_\Omega \sigma(v) : \varepsilon(v) \dx
	= 2\mu\, \|\varepsilon(v)\|^2 + \lambda\, (\nabla \cdot v \operatorname{id}, \varepsilon(v))
	= \mu\, \|v\|_\V^2 + (\mu + \lambda)\, \|\nabla\cdot v\|_\cHQ^2.
\end{align*}
This implies $c_a\ge\mu$. Together with $\|\nabla \cdot v\|_{\cHQ}^2 \le m\, \|v\|^2_{\V}$, where $m$ is again the dimension of the computational domain, this then yields~$\tilde \omega \le m\,\alpha^2 M / \mu$. In particular, this constant satisfies $d(v,q) \le \sqrt{\tilde \omega}\, \|v\|_a \|q\|_c$ with the problem-dependent norms~$\|\bullet\|^2_a \coloneqq a(\bullet,\bullet)$ and~$\|\bullet\|^2_c \coloneqq c(\bullet,\bullet)$. 

The choice of the parameter~$\tilde \omega$ is reasonable for general elliptic--parabolic problems. In the particular case of poroelasticity, however, we can exploit the special structure of the bilinear forms to improve this measure. Due to~$(\mu+\lambda)\, \|\nabla\cdot v\|_{\cHQ}^2 \le a(v,v)$, we obtain the estimate
\begin{equation}
	\label{eq:improvedOmega}	
	d(v,q) 
	\le \alpha\, \|\nabla \cdot v\|_{\cHQ}\|q\|_{\cHQ}
	\le \frac{\alpha \sqrt{M}}{\sqrt{\lambda+\mu}}\, \|v\|_a \|q\|_c. 
\end{equation}
This motivates to define the {\em coupling parameter} $\omega$ as
\[
	\omega 
	\coloneqq \frac{\alpha^2 M}{\mu+\lambda}.
\]

In the field of geomechanics, $M$ and the Lam{\'e} coefficients are usually of similar magnitude, while~$\alpha$ is bounded from above by~$1$. Hence, $\omega$ is expected to be in the range of~$1$. In medical applications such as brain matter simulations, there is a wide range of physical constants. Some make the assumption that~$1 / M \ll 1$, meaning that the material is almost incompressible, cf.~\cite{EliRT21}. This, in turn, implies that~$\omega$ is very large. Three particular examples for the coupling strength are collected in Table~\ref{tab:constants}.
\begin{table}
	\caption{Parameters for three materials, namely Westerly granite~\cite{DetC93} and shale~\cite{ZhaSMC19} (in combination with water), and brain matter~\cite{PieRV22}. }	
	\begin{center}
		\begin{tabular}{c c|c c c}
			parameter & unit & Westerly granite & shale & brain matter \\
			\hline
			\hline
			$\lambda$ & $\newton/\meter^2$ & \num{1.5e10} & \num{1.0e10} & \num{5.4e4} \\
			$\mu$ & $\newton/\meter^2$ & \num{1.5e10} & \num{1.0e10} & \num{5.5e2} \\
			$\alpha$ & & 0.47 & 0.92 & 1 \\
			$\kappa / \nu$ & $\meter^4/(\newton\second)$ & \num{4.0e-16} & \num{5.8e-14} & \num{1.6e-9} \\
			$M$ & $\newton/\meter^2$ & \num{7.64e10} & \num{9.5e10}& \num{2.6e3} \\
			\hline
			$\omega$ & & 0.56 & 4.02 & 0.05 \\
		\end{tabular}
	\end{center}
	\label{tab:constants}
\end{table}
%
%
\subsection{Spatial discretization}
We close this section with a short discussion of the spatial discretization using (conforming) finite elements and refer to~\cite{ErnM09} for further details. Within~\cite{ErnM09} it is suggested to use for the displacement piecewise polynomials of one degree higher than for the pressure. The most common choice is the lowest-order case which considers the discrete spaces 
\[
	V_h \coloneqq \V \cap [\calP_2(\calT)]^m, \qquad
	Q_h \coloneqq \Q \cap \calP_1(\calT)
\] 
with $\calT$ denoting a regular triangulation of the computational domain~$\Omega$ and $\calP_\ell(\calT)$ the space of piecewise polynomials of degree~$\ell$ corresponding to this mesh. Hence, the displacement is approximated by continuous piecewise quadratics, whereas continuous piecewise linears are used for the pressure. Using the same notation for the semi-discrete variables as in the continuous setting, the finite element discretization results in a system of the form 
\begin{align}
\label{eq:semiDiscretePoro}
	\begin{bmatrix} 0 & 0 \\ \Dmat & \Cmat \end{bmatrix}
	\begin{bmatrix} \dot u \\ \dot p \end{bmatrix}
	= 
	\begin{bmatrix} -\Amat & \phantom{-}\Dmat^T \\ 0 & -\Bmat \end{bmatrix}
	\begin{bmatrix} u \\ p \end{bmatrix}
	+ 
	\begin{bmatrix} \f \\ \g \end{bmatrix}.
\end{align}
Therein, the matrices~$\Amat\in\R^{n_u,n_u}$, $\Bmat, \Cmat\in\R^{n_p,n_p}$, and $\Dmat\in\R^{n_p,n_u}$ are the discrete versions of the bilinear forms $a$, $b$, $c$, and $d$, respectively. Moreover, the right-hand sides~$\hat \f$ and $\hat \g$ result in the vector-valued functions~$\f\colon [0,T]\to\R^{n_u}$ and~$\g\colon [0,T]\to\R^{n_p}$. 
\begin{remark}[Differential--algebraic structure]
Regardless of the particular method used for the spatial discretization, the leading matrix on the left-hand side of~\eqref{eq:semiDiscretePoro} is singular. Hence, the semi-discretized poroelasticity model equals a system of differential--algebraic equations, see also~\cite{AltMU21c}. 
The discrete version of the consistency condition for the initial data reads~$\Amat {u}^{0} - \Dmat^T p^0 = f(0)$ and will be assumed throughout this paper. 
\end{remark}
In the following, we are not restricted to the mentioned finite element scheme of piecewise quadratics combined with piecewise linears. We only assume a conforming discretization such that the semi-discrete system has the form~\eqref{eq:semiDiscretePoro} with matrices~$\Amat$, $\Bmat$, and~$\Cmat$ being symmetric and positive definite as well as~$\Dmat$ having full (row) rank. 
\begin{remark}[Coupling strength]
\label{rem:discreteOmega}
Considering a conforming finite element discretization, the discrete counterpart of estimate~\eqref{eq:improvedOmega} reads 
\begin{equation*}
	q^T\Dmat v
	\le \sqrt{\omega}\, \|v\|_\Amat \|q\|_\Cmat 
\end{equation*}
for all (coefficient) vectors $v\in\R^{n_u}$, $q\in\R^{n_p}$. This, in turn, implies the estimate
\[
	\rho(\Cmat^{-1}\Dmat\Amat^{-1}\Dmat^T) 
	\le \| \Cmat^{-1}\Dmat\Amat^{-1}\Dmat^T \|_\Cmat 
	\le \omega,
\]
which will be applied several times in the upcoming analysis. 
\end{remark}
%
%
\section{Time Stepping Schemes}\label{sec:methods}
This section is devoted to the temporal discretization of the semi-discrete system~\eqref{eq:semiDiscretePoro}. After a short survey of well-known time stepping methods, we introduce a novel approach in the intersection of semi-explicit and iterative schemes. 

For all time stepping schemes we consider an equidistant decomposition of the time interval~$[0,T]$ with step size~$\tau$. The resulting discrete time points are denoted by~$t^n = n\tau$ for $0 \le n \le N$. Moreover, approximations of~$u(t^n)$ and~$p(t^n)$ are denoted by~$u^n$ and $p^n$, respectively. In the same manner, we write~$\f^{n} \coloneqq \f(t^n) \in \R^{n_u}$ and~$\g^{n} \coloneqq \g(t^n) \in \R^{n_p}$.
%
%
\subsection{Implicit Euler scheme}\label{sec:methods:implEuler}
In the context of poroelasticity, a popular discretization method is the implicit Euler scheme for which the derivatives in time are replaced by a simple difference quotient. Multiplying the second equation by~$\tau$ and introducing 
\[
	\Ctau 
	\coloneqq \Cmat + \tau \Bmat, 
\]
this then leads to 
\begin{align}
\label{eq:implEuler}
	\begin{bmatrix} \Amat & -\Dmat^T \\ \Dmat & \phantom{-}\Ctau \end{bmatrix}
	\begin{bmatrix} u^{n+1} \\ p^{n+1} \end{bmatrix}
	= 
	\begin{bmatrix} 0 & 0 \\ \Dmat & \Cmat \end{bmatrix}
	\begin{bmatrix} u^{n} \\ p^{n} \end{bmatrix}
	+ 
	\begin{bmatrix} \f^{n+1} \\ \tau \g^{n+1} \end{bmatrix}.
\end{align}
This scheme is well-posed, since the leading matrix on the left-hand side is invertible for the given assumptions on $\Amat$, $\Bmat$, $\Cmat$, and $\Dmat$. Moreover, it is unconditionally stable. A detailed error analysis for the implicit Euler discretization in combination with a finite element discretization in space is given in~\cite{ErnM09}. As expected, this approach yields first-order convergence in time. 

The drawback of this fully implicit discretization is the fact that one needs to solve a large coupled system in each time step. Hence, the computation of an approximate solution can be very expensive, especially for three-dimensional applications. Moreover, established preconditioners for matrices of the form~$\Amat$ and~$\Ctau$ cannot be applied directly, cf.~\cite{LeeMW17}. 
%
%
\subsection{Iterative schemes}\label{sec:methods:iterative}
To avoid the solution of a large linear system in each time step, several decoupling strategies were introduced in the past couple of years. Here, decoupling means that the mechanics and the flow equation are solved sequentially. This not only implies that we have to solve two smaller rather than one large system but also facilitates the application of well-known preconditioners. 

Iterative schemes replace the solution of one implicit step by a sequence of decoupled solves. Hence, each time step contains an inner iteration, which is based on a matrix splitting of the form 
\[
	\begin{bmatrix} \Amat & -\Dmat^T \\ \Dmat & \phantom{-}\Ctau \end{bmatrix}
	= \Mmat - \Nmat. 
\]
The resulting time stepping scheme then reads as follows: first initialize~$u^{n+1}_0 \coloneqq u^n$ and~$p^{n+1}_0 \coloneqq p^n$ with the approximations of the previous time step. Then consider the inner iteration 
\[
	\Mmat \begin{bmatrix} u^{n+1}_{k+1} \\ p^{n+1}_{k+1} \end{bmatrix} 
	= \Nmat \begin{bmatrix} u^{n+1}_{k} \\ p^{n+1}_{k} \end{bmatrix} + \begin{bmatrix} \f^{n+1} \\ \tau \g^{n+1} + \Dmat u^{n} + \Cmat p^{n} \end{bmatrix} 
\]
for $k \ge 0$ until some convergence criteria is reached, e.g., until the residuum is smaller than a predefined tolerance. In the following, we specify the matrix splitting of four well-known iterative schemes. We would like to emphasize that all these schemes have in common that the matrix $M$ is block triangular. As such, the above system can be solved for~$u^{n+1}_{k+1}$ and~$p^{n+1}_{k+1}$ sequentially. 

\begin{itemize}
	\item Drained split \cite{ArmS92, SimM92, KimTJ11b}:
	\begin{equation*}
		\Mmat \coloneqq \begin{bmatrix} \Amat & 0 \\ \Dmat & \Ctau \end{bmatrix}, \qquad
		\Nmat \coloneqq \begin{bmatrix} 0 & \Dmat^T \\ 0 & 0 \end{bmatrix} 
	\end{equation*}
	\item Undrained split \cite{ArmS92, KimTJ11b}:
	\begin{equation*}
		\Mmat \coloneqq \begin{bmatrix} \Amat + \Dmat^T\Bmat^{-1}\Dmat & 0 \\ \Dmat & \Ctau \end{bmatrix}, \qquad
		\Nmat \coloneqq \begin{bmatrix} \Dmat^T\Bmat^{-1}\Dmat & \Dmat^T \\ 0 & 0 \end{bmatrix}
	\end{equation*}
	\item Fixed-strain split \cite{SetM98, KimTJ11a}: 
	\begin{equation*}
		\Mmat \coloneqq \begin{bmatrix} \Amat & -\Dmat^T \\ 0 & \phantom{-}\Ctau \end{bmatrix}, \qquad
		\Nmat \coloneqq \begin{bmatrix} 0 & 0 \\ -\Dmat & 0 \end{bmatrix}
	\end{equation*}
	\item Fixed-stress split \cite{SetM98, KimTJ11a}:
	\begin{equation*} 
		\Mmat \coloneqq \begin{bmatrix} \Amat & -\Dmat^T \\ 0 & \Ctau + \Dmat\Amat^{-1}\Dmat^T \end{bmatrix}, \qquad
		\Nmat \coloneqq \begin{bmatrix} 0 & 0 \\ -\Dmat & \Dmat\Amat^{-1}\Dmat^T \end{bmatrix}
	\end{equation*}
\end{itemize}

A drawback of these iterative methods is that the error analysis usually does not consider stability under a finite number of iterations. Additionally, only the undrained and fixed-stress split are unconditionally stable, but they require approximations of the matrices $\Dmat^T\Bmat^{-1}\Dmat$ and~$\Dmat\Amat^{-1}\Dmat^T$, respectively. Whether these methods converge then strongly depends on the choice of the respective tuning parameters, cf.~\cite{StoBKNR19}.
\begin{remark}
Indeed it can be shown, that the fixed-stress split exactly converges to the implicit Euler solution for two inner iteration steps if~$\Dmat\Amat^{-1}\Dmat^T$ is computed exactly. This, however, is not feasible in practice~\cite[Rem.~5]{KimTJ11a}. 
\end{remark}
%
%
\subsection{Semi-explicit Euler scheme}\label{sec:methods:semiexplEuler}
With the aim to circumvent an inner iteration but still allow a decoupling of the equations, the semi-explicit Euler scheme was introduced in~\cite{AltMU21}. This scheme is again based on the implicit Euler method with the modification that the pressure variable lags one time step behind. More precisely, we solve  
\[
	\begin{bmatrix} \Amat & 0 \\ \Dmat & \Ctau \end{bmatrix}
	\begin{bmatrix} u^{n+1} \\ p^{n+1} \end{bmatrix}
	= 
	\begin{bmatrix} 0 & \Dmat^T \\ \Dmat & \Cmat^{\phantom{T}} \end{bmatrix}
	\begin{bmatrix} u^{n} \\ p^{n} \end{bmatrix}
	+ 
	\begin{bmatrix} \f^{n+1} \\ \tau \g^{n+1} \end{bmatrix}.
\]
Note that this corresponds to the drained split scheme with a single inner iteration step. The analysis presented in~\cite{AltMU21} reveals that this is indeed a first-order scheme as long as the assumption~$\omega \le 1$ holds. Hence, this scheme is only applicable in the {\em weakly coupled} regime. 
%
%
\subsection{A novel iterative scheme with damping}\label{sec:methods:new}
This section is devoted to the introduction of a novel time stepping scheme which combines the benefits of the previous approaches. First, we are interested in a scheme which is also applicable for cases with~$\omega > 1$. Second, we desire a fixed amount of inner iteration steps without any termination condition, still assuring first-order convergence. In essence, the proposed scheme is a modification of the drained approach with an additional relaxation step, which depends on the coupling parameter~$\omega$.

Let $\gamma \coloneqq 2 / (2 + \omega)$ be the relaxation factor and $K \in \N$ the (fixed) number of inner iteration steps, which we link to the coupling parameter~$\omega$ later on. For each time step we initialize
\begin{subequations}
\label{eq:scheme}
\begin{align}
\label{eq:scheme:init}
	u^{n+1}_0 \coloneqq u^{n}, \qquad 
	p^{n+1}_0 \coloneqq p^{n}.
\end{align}
Then, for $k = 0, \dots, K-1$ we solve the (block triangular) system  
\begin{align}
\label{eq:scheme:hat}	
	\begin{bmatrix} \Amat & 0 \\ \Dmat & \Ctau \end{bmatrix}
	\begin{bmatrix} {\hat u}^{n+1}_{k+1} \\ {\hat p}^{n+1}_{k+1} \end{bmatrix}
	= 
	\begin{bmatrix} 0 & \Dmat^T \\ 0 & 0 \end{bmatrix}
	\begin{bmatrix} u^{n+1}_k \\ p^{n+1}_k \end{bmatrix}
	+
	\begin{bmatrix} \f^{n+1} \\ \tau \g^{n+1} + \Dmat u^n + \Cmat p^n \end{bmatrix}
\end{align}
for~$\hat u^{n+1}_{k+1}$ and~$\hat p^{n+1}_{k+1}$. Except for the last step we dampen the pressure variable, i.e., for $k = 0, \dots, K-2$ we set  
\begin{equation} 
\label{eq:scheme:relaxation}
	u_{k+1}^{n+1}
 	\coloneqq \hat{u}_{k+1}^{n+1}, \qquad
	p_{k+1}^{n+1} 
	\coloneqq \gamma\, \hat{p}_{k+1}^{n+1} + (1 - \gamma)\, p_{k}^{n+1}.
\end{equation}
Finally, the approximations at time $t=t^{n+1}$ are given by 
\begin{equation}
\label{eq:scheme:final}
	u^{n+1} 
	\coloneqq u_K^{n+1}
	\coloneqq \hat{u}_K^{n+1},\qquad 
	p^{n+1} 
	\coloneqq p_K^{n+1}
	\coloneqq \hat{p}_K^{n+1}.
\end{equation}
\end{subequations}
We would like to emphasize that there is no relaxation for the deformation necessary, since the previous (inner) iterate $u^{n+1}_k$ does not appear in~\eqref{eq:scheme:hat}. Moreover, the relaxation is skipped in the final step and the scheme is well-posed, since the leading matrix in~\eqref{eq:scheme:hat} is invertible.  
\begin{remark}[No convergence with final relaxation] 
\label{rem:final-relaxation}
If we would include the relaxation also in the final step, then one can observe that the scheme does not converge. This is no surprise, since this modification would lead to a direct linear influence of~$p^n$ on~$p^{n+1}$. 
\end{remark}
The computation of one time step is summarized in Algorithm~\ref{alg:oneTimeStep}. 
\begin{algorithm}
	\caption{One time step of scheme~\eqref{eq:scheme}}
	\label{alg:oneTimeStep}
	\begin{algorithmic}
		\State {\bf Input}: current approximations $u^n\in\R^{n_u}$, $p^n\in\R^{n_p}$, step size $\tau$
		\vspace{0.3em}
		\State $p \gets p^n$
		\State $r \gets \tau g^{n+1} + \Dmat u^n + \Cmat p^n$
		\vspace{0.3em}
		\For{$k = 0 : K-2$}
		\State $u \gets \Amat^{-1} (f^{n+1} + \Dmat^Tp )$
		\State $\hat p \gets \Ctau^{-1} (r - \Dmat u)$
		\State $p \gets \gamma\, \hat p + (1 - \gamma)\, p$ 
		\EndFor
		\vspace{0.3em}
		\State $u^{n+1} \gets \Amat^{-1} (f^{n+1} + \Dmat^Tp )$ 
		\State $p^{n+1} \gets \Ctau^{-1} (r - \Dmat u^{n+1})$ 
	\end{algorithmic}
\end{algorithm}
\begin{remark}[Interpretation as matrix splitting] \label{rem:matrix-splitting}
Apart from the last step, the inner iteration of the proposed scheme~\eqref{eq:scheme} can be written as the matrix splitting 
\[
	\Mmat \coloneqq \frac{1}{\gamma}\, \begin{bmatrix} \Amat & 0 \\ \Dmat & \Ctau\end{bmatrix}, \qquad
	\Nmat \coloneqq \begin{bmatrix} 0 & \Dmat^T \\ 0 & 0 \end{bmatrix} +
	\frac{1-\gamma}{\gamma}\, \begin{bmatrix} \Amat & 0 \\ \Dmat & \Ctau\end{bmatrix}.
\]
For the convergence of the inner iteration, one needs to analyze the spectral radius of~$\Mmat^{-1}\Nmat$. By Remark~\ref{rem:discreteOmega} and the positive definiteness of $\Ctau^{-1}$ and $\Dmat\Amat^{-1}\Dmat^T$ we recognize that the spectrum of~$\Ctau^{-1}\Dmat\Amat^{-1}\Dmat^T$ lies in the interval~$[0, \omega]$, leading to 
\begin{align*}
	\rho\big( \Mmat^{-1}\Nmat \big) 
	&= \rho\left( \gamma \begin{bmatrix} \Amat^{-1} & 0 \\ -\Ctau^{-1}\Dmat\Amat^{-1} & \Ctau^{-1} \end{bmatrix} \begin{bmatrix} 0 & \Dmat^T \\ 0 & 0 \end{bmatrix} + (1-\gamma) \begin{bmatrix} \Imat & 0 \\ 0 & \Imat \end{bmatrix} \right) \\
	&= \rho \left( \begin{bmatrix} (1-\gamma)\Imat & \gamma\Amat^{-1}\Dmat^T \\ 0 & -\gamma \, \Ctau^{-1}\Dmat\Amat^{-1}\Dmat^T + (1-\gamma)\, \Imat \end{bmatrix} \right)\\
	&= \max \big\{1-\gamma, \rho\big( -\gamma\, \Ctau^{-1}\Dmat\Amat^{-1}\Dmat^T + (1-\gamma)\, \Imat \big) \big\} \\
	&\le \max \{1-\gamma, |-\gamma\omega + (1-\gamma)|, |0 + (1 - \gamma)|\} \\
	&= \max \{1-\gamma, \gamma\omega - (1-\gamma)\}.  
\end{align*}
Note that for $\gamma = 1$, which corresponds to the drained split, we obtain~$\rho(\Mmat^{-1}\Nmat) \le \omega$ and, hence, the condition~$\omega < 1$. For the proposed scheme, on the other hand, we have~$\gamma = 2 / (2 + \omega)$, for which this bound takes its minimum $\rho(\Mmat^{-1}\Nmat) \le \omega / (\omega + 2) < 1$. Hence, we converge to the implicit Euler scheme for~$K \to \infty$. 
\end{remark}
We would like to conclude this section with two particular examples of the proposed time stepping scheme. 
For $K=1$, there is no relaxation step and we recover the semi-explicit Euler scheme from Section~\ref{sec:methods:semiexplEuler} or the drain split scheme with a single inner iteration step. 
The procedure for the case $K=2$ is summarized in Algorithm~\ref{alg:K2}. 
\begin{algorithm}
	\caption{Time stepping scheme~\eqref{eq:scheme} for $K = 2$}
	\label{alg:K2}
	\begin{algorithmic}
		\State {\bf Input}: (consistent) initial data $u^0\in\R^{n_u}$, $p^0\in\R^{n_p}$, time horizon~$T$, step size $\tau$
		\vspace{0.3em}
		\State $N \gets T/\tau$
		\State $u \gets u^0$
		\State $p \gets p^0$
		\State $n \gets 0$
		\vspace{0.3em}
		\While{$n \le N-1$}
		\State $r \gets \tau g^{n+1} + \Dmat u + \Cmat p$
		\State $u^{n+\frac{1}{2}} \gets \Amat^{-1} (f^{n+1} + \Dmat^Tp)$
		\State $p^{n+\frac{1}{2}} \gets \Ctau^{-1} (r - \Dmat u^{n+\frac{1}{2}})$
		\State $u \gets \Amat^{-1} \big(f^{n+1} + \Dmat^T( \gamma\, p^{n+\frac{1}{2}} + (1 - \gamma)\, p) \big)$
		\State $p \gets \Ctau^{-1} \big(r - \Dmat u\big)$ 
		\State $n \gets n + 1$
		\EndWhile
		%
	\end{algorithmic}
\end{algorithm}
%
%
\section{Convergence Analysis}\label{sec:convergence}
In this section, we prove first-order convergence of the semi-explicit time stepping scheme~\eqref{eq:scheme}. For this, we first present some kind of stability result for~$K$ being sufficiently large in terms of the coupling parameter~$\omega$. For the convergence analysis, we make the following assumption on the initial data.
\begin{assumption} 
\label{ass:initialStaticPressure}
We assume that $\Bmat p^0 = g(0) - \Dmat\Amat^{-1}\dot f(0) + \calO(\tau)$.  
\end{assumption}
\begin{remark}
The condition in Assumption~\ref{ass:initialStaticPressure} states that the pressure $p$ satisfies~$\dot p(0)= \calO(\tau)$. Although this does not seem to be a natural assumption at first sight, this condition is reasonable. For the real-word application of brain tissue~\cite{JuCLT20}, which we consider in Section~\ref{sec:numerics:brain} below, this means that the force $g$ should not apply instantaneously. 
Moreover, numerical experiments indicate that this assumption is only technical and not needed in practice. Also in the special case $K = 1$, which corresponds to the semi-explicit Euler scheme~\cite{AltMU21}, no such condition is needed to prove first-order convergence. 
\end{remark}
\begin{proposition}
\label{prop:convergence-splitting}
Let the right-hand sides of the semi-discrete problem~\eqref{eq:semiDiscretePoro} satisfy the smoothness conditions~$f \in C^2([0,T], \R^{n_u})$ and~$g \in C^1([0,T], \R^{n_p})$. Further assume consistent initial data, Assumption~\ref{ass:initialStaticPressure}, and  
\begin{equation}
\label{eq:iterationBound}
	\frac{\omega^K}{(2+\omega)^{K-1}}
	< 1.
\end{equation}
Then the iterates of the scheme~\eqref{eq:scheme} are stable in the sense that 
\begin{subequations}
\begin{align}
	\big\| p^{n+1} - p_{K-1}^{n+1} \big\|_{\Ctau} 
	&\le C\, e^T \tau, \label{eq:convergence-splitting:stability1} \\
	\big\| (p^{n+1} - p_{K-1}^{n+1}) - (p^n - p_{K-1}^n) \big\|_{\Ctau} 
	&\le C\, e^T \tau^2, \label{eq:convergence-splitting:stability2}
\end{align}
\end{subequations}
where $C>0$ is a constant that depends only on the $C^2([0,T], \R^{n_u})$-norm of $f$, the $C^1([0,T], \R^{n_p})$-norm of~$g$, and the coupling parameter $\omega$. 
\end{proposition}
\begin{proof}
Consider the differences appearing in the first step of the inner iteration, 
\[ 
	\delta_u^{n+1} 
	\coloneqq \hat{u}^{n+1}_1 - u^n, \qquad 
	\delta_p^{n+1} 
	\coloneqq \hat{p}^{n+1}_1 - p^n.  
\]
We first prove the linear dependence 
\begin{equation} 
\label{eq:convergence-splitting:linear}
	\hat{p}_{k+1}^{n+1} - p^{n+1}_k 
	= \Smat^k\delta_p^{n+1}
\end{equation}
for $0 \le k \le K-1$, where $\Smat = \gamma\, \Tmat + (1 - \gamma)\, \mathbf{I}$ with $\Tmat = -\Ctau^{-1}\Dmat \Amat^{-1}\Dmat^T$ and~$\mathbf{I}$ denoting the identity matrix. For this, observe that by~\eqref{eq:scheme:hat} we get for all $1 \le k \le K-1$,  
\[
	\hat{p}_k^{n+1} 
	= \Tmat p_{k-1}^{n+1} + r^{n+1}, 
\]
where $r^{n+1} = \Ctau^{-1}(\tau g^{n+1} + \Dmat u^n + \Cmat p^n - \Dmat\Amat^{-1} f^{n+1})$. Combined with~\eqref{eq:scheme:relaxation} this shows 
\begin{equation} 
\label{eq:convergence-splitting:linear1}
	p_k^{n+1} 
	= \Smat p_{k-1}^{n+1} + \gamma\, r^{n+1}.
\end{equation}
On the other hand, we observe that 
\begin{align}
	\hat{p}_{k+1}^{n+1}
	&= \Tmat\, \big(\gamma\, \hat{p}_k^{n+1} + (1 - \gamma)\, p_{k-1}^{n+1} \big) + r^{n+1} \notag \\
	&= \gamma\, \Tmat\hat{p}_k^{n+1} + (1 - \gamma) \big(\Tmat p_{k-1}^{n+1} + r^{n+1} \big) + \gamma\, r^{n+1} 
	= \Smat\hat{p}_k^{n+1} + \gamma\, r^{n+1}. \label{eq:convergence-splitting:linear2}
\end{align}
Taking the difference of \eqref{eq:convergence-splitting:linear2} and~\eqref{eq:convergence-splitting:linear1}, we get 
\[
	\hat{p}_{k+1}^{n+1} - p_k^{n+1}
	= \Smat\big( \hat{p}_k^{n+1} - p_{k-1}^{n+1} \big).
\]
A repeated application of~\eqref{eq:convergence-splitting:linear1} and~\eqref{eq:convergence-splitting:linear2} then yields 
\[
	\hat{p}^{n+1}_{k+1} - p_k^{n+1}
	= \Smat\,\big( \hat{p}_k^{n+1} - p_{k-1}^{n+1} \big)
	= \dots 
	= \Smat^k \big( \hat{p}_{1}^{n+1} - p_{0}^{n+1} \big)
	= \Smat^k \delta_p^{n+1} 
\]
and hence~\eqref{eq:convergence-splitting:linear}. Next, we study the relation of~$\delta_p^{n+1}$ and~$p^{n+1} - p^n$, leading to 
\begin{align}
	p^{n+1} - p^n
	&= p^{n+1} - p_{K-1}^{n+1} + \sum\nolimits_{k=0}^{K-2} \big( p_{k+1}^{n+1} - p_k^{n+1} \big) \notag \\
	&= \Smat^{K-1}\delta_p^{n+1} + \sum\nolimits_{k=0}^{K-2} \big( \gamma\, \hat{p}_{k+1}^{n+1} + (1 - \gamma)\, p_k^{n+1} - p_k^{n+1} \big) \notag \\
	&= \Smat^{K-1}\delta_p^{n+1} + \gamma \sum\nolimits_{k=0}^{K-2} \big( \hat{p}_{k+1}^{n+1} - p_k^{n+1} \big) \notag \\
	&= \Big(\Smat^{K-1} + \gamma \sum\nolimits_{k=0}^{K-2} \Smat^k\Big)\, \delta_p^{n+1}
	\eqqcolon \tilde \Smat\, \delta_p^{n+1}. \label{eq:convergence-splitting:deltap1}
\end{align}
For the matrix~$\tilde \Smat$ one can show that its norm is bounded by a constant $0 < C_\Smat \le 1$. To see this, we first note that $\|\Smat x\|_{\Ctau} \le 1-\gamma$, cf.~Remark~\ref{rem:matrix-splitting}. 
Then, applying the geometric series, we can estimate 
\begin{align*}
	\|\tilde \Smat x\|_{\Ctau} 
	&\le \Big( (1 - \gamma)^{K-1} + \gamma\, \sum\nolimits_{k=0}^{K-2} (1 - \gamma)^k \Big) \|x\|_{\Ctau} \\
	&= \Big( (1 - \gamma)^{K-1} + \gamma\, \frac{1 - (1 - \gamma)^{K-1}}{1 - (1 - \gamma)} \Big) \|x\|_{\Ctau}
	= \|x\|_{\Ctau}.
\end{align*}
With~\eqref{eq:convergence-splitting:linear} we note by~\eqref{eq:scheme:hat} that the values $\delta_u^{n+1}$ and $\delta_p^{n+1}$ solve for $n\ge1$, 
\begin{subequations}
\label{eq:convergence-splitting:system1}	
\begin{align}
	\Amat\delta_u^{n+1} - \Dmat^T \Smat^{K-1} \delta_p^n 
	&= \f^{n+1} - \f^n, \label{eq:convergence-splitting:system1:a} \\
	\Dmat\delta_u^{n+1} + \Cmat\delta_p^{n+1} + \tau \Bmat\delta_p^{n+1} 
	&= \tau g^{n+1} - \tau \Bmat p^n. \label{eq:convergence-splitting:system1:b}
\end{align}
\end{subequations}
Due to the assumed consistency of the initial data, i.e., $\Amat u^{0} - \Dmat^T p^0 = \f^{0}$, \eqref{eq:convergence-splitting:system1} still holds for~$n=0$ if we set $\delta_p^0 \coloneqq 0$. The right-hand sides can be bounded using the Taylor approximation~$\f^{n+1} - \f^n = \tau\dot{f}(\xi_f^{n+1})$ for some $\xi_f^{n+1} \in [t^n, t^{n+1}]$. Solving the first equation~\eqref{eq:convergence-splitting:system1:a} for $\delta_u^{n+1}$ and inserting it into the second leads to 
\begin{equation} \label{eq:convergence-splitting:step}
	\delta_p^{n+1} 
	= \Tmat\Smat^{K-1}\delta_p^n
	- \tau\, \Ctau^{-1} \Bmat p^n
	- \tau\, \Ctau^{-1} \big[ \Dmat\Amat^{-1} \dot{f}(\xi_f^{n+1}) - g^{n+1} \big].
\end{equation}
It is easy to see that the term in brackets is uniformly bounded by a constant $\Cfg = \Cfg(\omega, \|f\|_{C^1([0,T];\R^{n_u})}, \|g\|_{C^0([0,T];\R^{n_p})})$. For the term containing $p^n$, this is not so obvious. Instead we now show the boundedness of~$p^n$ and~$\delta_p^n$ simultaneously. To estimate the effect of the arising terms in the ${\Ctau}$-norm, we observe that 
%
%
\begin{align*}
	\|\Tmat x\|_{\Ctau} 
	\le \omega\, \|x\|_{\Ctau}
	\qquad\text{and}\qquad
	\|\Smat x\|_{\Ctau} \le (1 - \gamma)\, \|x\|_{\Ctau}
\end{align*}
for any $x \in \R^{n_p}$, cf.~Remarks~\ref{rem:discreteOmega} and~\ref{rem:matrix-splitting}. Hence, we can estimate 
\begin{align*}
	\|\delta_p^{n+1}\|_{\Ctau}
	&\le \|\Tmat\Smat^{K-1}\delta_p^n\|_{\Ctau} 
	+ \tau\, \|\Ctau^{-1}\Bmat p^n\|_{\Ctau} 
	+ \tau\, \Cfg \\
	&\le \omega\, (1 - \gamma)^{K-1}\|\delta_p^n\|_{\Ctau} + \tau\, C_p \|p^n\|_{\Ctau} + \tau\, \Cfg
\end{align*}
for $C_p \coloneqq C_b/(c_c + \tau c_b)$. Recall that~$\gamma$ is fixed in terms of the coupling parameter, namely $\gamma = 2 / (2 + \omega)$. As such, the assumed bound~\eqref{eq:iterationBound} gives
\[
	\omega\, (1-\gamma)^{K-1}
	= \frac{\omega^K}{(2+\omega)^{K-1}}
	< 1. 
\]
Applying the above inequality iteratively, we thus obtain
\begin{align}
	\|\delta_p^{n+1}\|_{\Ctau}
	&\le \big( \omega\, (1-\gamma)^{K-1} \big)^n\, \|\delta_p^0\|_{\Ctau} 
	+ \tau\sum_{\ell=0}^{n-1} \big(\omega\, (1-\gamma)^{K-1}\big)^\ell \big(C_p\, \|p^{n-\ell}\|_{\Ctau} + \Cfg \big) \notag \\
	&\le \tau\, \frac{1}{1 - \omega\,(1-\gamma)^{K-1}}\, \big(\Cfg + C_p\, \max_{\ell\le n} \|p^\ell\|_{\Ctau} \big), \label{eq:convergence-splitting:deltap2}
\end{align}
where we bounded the sum by the geometric series. We combine this with~\eqref{eq:convergence-splitting:deltap1} to get 
\[
	\|p^{n+1}\|_{\Ctau}
	\le \|p^n\|_{\Ctau} + C_\Smat\, \|\delta_p^{n+1}\|_{\Ctau}
	\le (1 + \tau\, C_\Smat' C_p) \max_{\ell\le n} \|p^\ell\|_{\Ctau} + \tau\, C_\Smat'\Cfg
\]
for $C_\Smat'= \frac{C_\Smat}{1 - \omega\, (1-\gamma)^{K-1}}$. Using the discrete 
Gr\"onwall lemma, we finally get 
\[
	\|p^n\|_{\Ctau} 
	\le e^{t^n C_\Smat' C_p} \big(\|p^0\|_{\Ctau} + t^n\, C_\Smat' \Cfg \big).
\]
In other words, the iterates are bounded in terms of~$T$, $\omega$, the initial data, and the right-hand sides, but independent of~$\tau$. Inserting this estimate into~\eqref{eq:convergence-splitting:deltap2} and using~\eqref{eq:convergence-splitting:linear} shows the first claim \eqref{eq:convergence-splitting:stability1}.

For the second claim, we consider once more~\eqref{eq:convergence-splitting:system1}, leading to 
\begin{subequations}
\begin{align}
	\Amat(\delta_u^{n+1} - \delta_u^n) - \Dmat^T \Smat^{K-1} (\delta_p^n - \delta_p^{n-1}) 
	&= \f^{n+1} - 2\f^n + \f^{n-1}, \label{eq:convergence-splitting:system2:a} \\
	\Dmat(\delta_u^{n+1} - \delta_u^n) + \Cmat(\delta_p^{n+1} - \delta_p^n) + \tau \Bmat(\delta_p^{n+1} - \delta_p^n) 
	&= \tau\, (\g^{n+1} - \g^n) - \tau \Bmat (p^n - p^{n-1}). \label{eq:convergence-splitting:system2:b}
\end{align}
\end{subequations}
Considering~\eqref{eq:convergence-splitting:step} for $n = 0$, we get by $\delta_p^0 = 0$ and Assumption~\ref{ass:initialStaticPressure}, 
\[
	\delta_p^{1} 
	= 
	\tau\, \Ctau^{-1} \big[ g^{1} - \Dmat\Amat^{-1} \dot{f}(\xi_f^{1}) - \Bmat p^0\big]
	= \tau^2\, \Ctau^{-1} \big[ \dot g({\xi}) - \Dmat\Amat^{-1}\ddot f(\zeta) \big] + \calO(\tau^2)
\]	
for some $\xi, \zeta \in [0, \tau]$. Hence, we have $\delta_p^{1} = \mathcal{O}(\tau^2)$. 

Using a Taylor approximation, we can now replace the right-hand side of \eqref{eq:convergence-splitting:system2:a} by $\tau^2 \ddot{f}(\xi_f^{\prime n+1})$ for some $\xi_f^{\prime n+1} \in [t^{n-1}, t^{n+1}]$. In~\eqref{eq:convergence-splitting:system2:b} we obtain $\tau^2\dot{g}(\xi_g^{n+1}) + \tau \Bmat \tilde \Smat \delta_p^n$ for some $\xi_g^{n+1} \in [t^n, t^{n+1}]$. Inequality~\eqref{eq:convergence-splitting:stability2} then follows using a similar procedure as for obtaining~\eqref{eq:convergence-splitting:deltap2} and applying this particular estimate to bound~$\delta_p^n$. More precisely, we get 
\begin{align*}
	\|\delta_p^{n+1} - \delta_p^n\|_{\Ctau}
	&\le \big( \omega\, (1-\gamma)^{K-1} \big)^n\, \|\delta_p^1\|_{\Ctau} 
	+ \tau^2\, \frac{C}{1 - \omega\,(1-\gamma)^{K-1}}  
\end{align*}
with a constant $C$ depending on $\Cfg$ and $C_p$. Finally, \eqref{eq:convergence-splitting:linear} with $k=K-1$ yields
\[
	(p^{n+1} - p_{K-1}^{n+1}) - (p^n - p_{K-1}^n)
	= \Smat^{K-1} \big( \delta_p^{n+1} - \delta_p^{n} \big), 
\]
which directly leads to the assertion. 
\end{proof}
We now present the main result of this paper, namely the convergence proof for the proposed scheme~\eqref{eq:scheme}. For this, we combine ideas used in \cite[Lem.~3.2]{AltMU21} for the last step of the inner iteration with the previous result obtained in Proposition~\ref{prop:convergence-splitting}.
\begin{theorem}[First-order convergence]
\label{thm:general-convergence} 
Let the solution $(u, p)$ of the semi-discrete problem~\eqref{eq:semiDiscretePoro} satisfy the smoothness conditions~$u \in C^2([0,T], \R^{n_u})$ and~$p \in C^2([0,T], \R^{n_p})$. Moreover, consider the assumptions of Proposition~\ref{prop:convergence-splitting} including~\eqref{eq:iterationBound}, i.e., $\omega^K / (2+\omega)^{K-1} < 1$. Then scheme~\eqref{eq:scheme} converges with order one. More precisely, we have 
\begin{align*}
	\|u^n - u(t^n)\|_\Amat^2 + \|p^n - p(t^n)\|_\Bmat^2 
	\le C\, \tau^2
\end{align*}
with a positive constant~$C>0$ depending on the solution, the right-hand sides, the time horizon, and the material parameters. 
\end{theorem}
\begin{proof}
We denote the errors between the semi-discrete solution and the outcome of~\eqref{eq:scheme} by 
\[ 
	\eta_u^{n} \coloneqq u^{n} - u(t^{n}), \qquad
	\eta_p^{n} \coloneqq p^{n} - p(t^{n}).  
\]
Now let $1 \le n \le N$. Using the Taylor approximations 
\begin{align*}
	u(t^n) 
	= u(t^{n+1}) - \tau\dot{u}(t^{n+1}) + \tfrac{1}{2}\, \tau^2 \ddot{u}(\xi_u^{n+1}),\quad 
	p(t^n) 
	= p(t^{n+1}) - \tau\dot{p}(t^{n+1}) + \tfrac{1}{2}\, \tau^2 \ddot{p}(\xi_p^{n+1}) 
\end{align*}
for some $\xi_u^{n+1}, \xi_p^{n+1} \in [t^n, t^{n+1}]$, we obtain from equations~\eqref{eq:semiDiscretePoro} and~\eqref{eq:scheme:hat} for $k = K-1$ that 
\begin{subequations} \label{eq:general-convergence:system}
\begin{align}
	\Amat\eta_u^{n+1} - \Dmat^T\eta_p^{n+1} 
	&= \Dmat^T(p_{K-1}^{n+1} - p^{n+1}), \\
	\Dmat(\eta_u^{n+1} - \eta_u^n) + \Cmat(\eta_p^{n+1} - \eta_p^n) + \tau \Bmat\eta_p^{n+1} 
	&= \tfrac{1}{2}\, \tau^2 \big[\Dmat\ddot{u}(\xi_u^{n+1}) + \Cmat\ddot{p}(\xi_p^{n+1}) \big].
\end{align}
\end{subequations}
Taking the difference with the previous time step in the first equation results in
\[ 
	\Amat(\eta_u^{n+1} - \eta_u^n) - \Dmat^T(\eta_p^{n+1} - \eta_p^n) 
	= \Dmat^T \big(p_{K-1}^{n+1} - p^{n+1} - (p_{K-1}^n - p^n) \big). 
\]
Now, multiplying the latter two equations from the left with $(\eta_u^{n+1} - \eta_u^n)^T$ and $(\eta_p^{n+1} - \eta_p^n)^T$, respectively, we obtain for their sum 
\begin{align*}
	&\|\eta_u^{n+1} - \eta_u^n\|_\Amat^2 + \|\eta_p^{n+1} - \eta_p^n\|_\Cmat^2 + \tfrac{1}{2}\, \tau\, \big(\|\eta_p^{n+1}\|_\Bmat^2 - \|\eta_p^n\|_\Bmat^2\big) \\
	\le\ &\|\eta_u^{n+1} - \eta_u^n\|_\Amat^2 + \|\eta_p^{n+1} - \eta_p^n\|_\Cmat^2 + \tau\,(\eta_p^{n+1} - \eta_p^n)^T \Bmat\eta_p^{n+1} \\
	\le\ &\big(p_{K-1}^{n+1} - p^{n+1} - (p_{K-1}^n - p^n)\big)^T\Dmat(\eta_u^{n+1} - \eta_u^n)\\
	&\hspace{7em} - \tfrac{1}{2}\, \tau^2 (\eta_p^{n+1} - \eta_p^n)^T \big[\Dmat\ddot{u}(\xi_u^{n+1}) + \Cmat\ddot{p}(\xi_p^{n+1})\big].
\end{align*}
Here, we have used the identity
\begin{equation} \label{eq:general-convergence:identity}
	2\, (\eta_p^{n+1} - \eta_p^n)^T \Bmat\eta_p^{n+1} 
	= \|\eta_p^{n+1}\|_\Bmat^2 - \|\eta_p^n\|_\Bmat^2 + \|\eta_p^{n+1} - \eta_p^n\|_\Bmat^2.
\end{equation}
Using Remark~\ref{rem:discreteOmega}, Proposition~\ref{prop:convergence-splitting}, and Young's inequality, the first summand of the right-hand side can be bounded by 
\begin{align*}
	\big(p_{K-1}^{n+1} - p^{n+1}& - (p_{K-1}^n - p^n)\big)^T\Dmat(\eta_u^{n+1} - \eta_u^n) \\
	&\qquad\le \sqrt{\omega}\, \|p^{n+1} - p^{n+1}_{K-1} - (p^n - p^n_{K-1})\|_\Cmat \, \|\eta_u^{n+1} - \eta_u^n\|_\Amat \\ 
	&\qquad\le \tfrac12\, \omega\, \tau^4\, C_\text{Prop}^2 e^{2T} + \tfrac12\, \|\eta_u^{n+1} - \eta_u^n\|_\Amat^2 
\end{align*}
with $C_\text{Prop}$ being the constant from~\eqref{eq:convergence-splitting:stability2}. Similarly, we obtain for the second summand  
\begin{align*}
	\tfrac{1}{2}\, \tau^2 (\eta_p^{n+1} - \eta_p^n)^T &\big[\Dmat\ddot{u}(\xi_u^{n+1}) + \Cmat\ddot{p}(\xi_p^{n+1}) \big] \\
	&\qquad\le \tfrac{1}{2}\, \tau^2 \|\eta_p^{n+1} - \eta_p^n\|_\Cmat
	\big( \sqrt{\omega}\, \|\ddot{u}(\xi_u^{n+1})\|_\Amat + \| \ddot{p}(\xi_p^{n+1})\|_\Cmat  \big)  \\ 
	&\qquad\le \tfrac18\, \tau^4 \big( \omega\, \|\ddot{u}\|^2_{L^\infty(v)} + \|\ddot{p}\|_{L^\infty(\Cmat)}^2 \big) + \|\eta_p^{n+1} - \eta_p^n\|_\Cmat^2.
\end{align*}
Absorbing both terms~$\tfrac12\, \|\eta_u^{n+1} - \eta_u^n\|_\Amat^2$ and $\|\eta_p^{n+1} - \eta_p^n\|_\Cmat^2$ and dividing by $\tau/2$, we obtain 
\begin{equation} \label{eq:general-convergence:single-step}
	\|\eta_p^{n+1}\|_\Bmat^2 - \|\eta_p^n\|_\Bmat^2 + \tfrac1\tau\, \|\eta_u^{n+1} - \eta_u^n\|_\Amat^2
	\le \omega\, \tau^3 C_\text{Prop}^2 e^{2T}  
	+ \tfrac14\, \tau^3 \big( \omega\, \|\ddot{u}\|^2_{L^\infty(\Amat)} + \|\ddot{p}\|_{L^\infty(\Cmat)}^2 \big).
\end{equation}
Summation over $n$ then gives 
\begin{align*}
	\|\eta_p^n\|_\Bmat^2 + \frac1\tau\, \sum_{j=0}^{n-1} \|\eta_u^{j+1} - \eta_u^j\|_\Amat^2
	\le T\, \tau^2\, \Big( \omega\, C_\text{Prop}^2 e^{2T} + \tfrac14\, \omega\, \|\ddot{u}\|^2_{L^\infty(\Amat)} + \tfrac14\, \|\ddot{p}\|_{L^\infty(\Cmat)}^2 \Big)
\end{align*}
showing the claimed bound of~$\eta_p^n$.

For the bound of $\eta_u^n$ we consider again system~\eqref{eq:general-convergence:system}. This time, however, using~$\eta_u^{n+1} - \eta_u^n$ and~$\eta_p^{n+1}$ as test functions, respectively. The sum then reads
\begin{align*}
	&(\eta_u^{n+1} - \eta_u^n)^T\Amat\eta_u^{n+1} +
	(\eta_p^{n+1})^T\Cmat(\eta_p^{n+1} - \eta_p^n) + \tau\, \|\eta_p^{n+1}\|_\Bmat^2 \\
	&\qquad= (\eta_u^{n+1} - \eta_u^n)^T\Dmat^T(p_{K-1}^{n+1} - p^{n+1})
	+ \tfrac{1}{2}\, \tau^2 (\eta_p^{n+1})^T\big[\Dmat\ddot{u}(\xi_u^{n+1}) + \Cmat\ddot{p}(\xi_p^{n+1}) \big]. 
\end{align*}
By the identity~\eqref{eq:general-convergence:identity}, this can be transformed to  
\begin{align*}
	&\|\eta_u^{n+1}\|_\Amat^2 - \|\eta_u^n\|_\Amat^2 + \|\eta_u^{n+1} - \eta_u^n\|_\Amat^2 +
	\|\eta_p^{n+1}\|_\Cmat^2 - \|\eta_p^n\|_\Cmat^2 + \|\eta_p^{n+1} - \eta_p^n\|_\Cmat^2 +
	2\tau\, \|\eta_p^{n+1}\|_\Bmat^2 \\
	&\qquad= 2\, (\eta_u^{n+1} - \eta_u^n)^T\Dmat^T(p_{K-1}^{n+1} - p^{n+1})
	+ \tau^2 (\eta_p^{n+1})^T\big[\Dmat\ddot{u}(\xi_u^{n+1}) + \Cmat\ddot{p}(\xi_p^{n+1}) \big].
\end{align*}
Using once more Young's inequality, we bound the right-hand side by 
\begin{align*}
	\tfrac1\tau\, \|\eta_u^{n+1} - \eta_u^n\|_\Amat^2 + \omega\, \tau\, \|p_{K-1}^{n+1} - p^{n+1}\|_\Cmat^2
	+ \tfrac14 C_\text{em} \tau^3 \big( \omega\, \|\ddot{u}\|^2_{L^\infty(\Amat)} + \|\ddot{p}\|_{L^\infty(\Cmat)}^2 \big) 
	+ 2\tau\, \|\eta_p^{n+1}\|_\Bmat^2, 
\end{align*}
where $C_\text{em} = C_c C^2_{\V\hook\cHV} / c_b$ includes the continuity constant of the embedding~$\V\hook\cHV$. Applying \eqref{eq:general-convergence:single-step} and inequality~\eqref{eq:convergence-splitting:stability1} from Proposition~\ref{prop:convergence-splitting}, this is bounded by 
\begin{align*}
	\tau^3 \left[
	2\omega\, C_\text{Prop}^2\, e^{2T} 
	+ \tfrac14\, (1+C_\text{em}) \big( \omega\, \|\ddot{u}\|^2_{L^\infty(\Amat)} + \|\ddot{p}\|_{L^\infty(\Cmat)}^2 \big)
	\right] 
	+ 2\tau\, \|\eta_p^{n+1}\|_\Bmat^2.
\end{align*}
By absorbing $2\tau\, \|\eta_p^{n+1}\|_\Bmat^2$ and dropping the terms $\|\eta_u^{n+1} - \eta_u^n\|_\Amat^2$ and~$\|\eta_p^{n+1} - \eta_p^n\|_\Cmat^2$ on the left-hand side, the full inequality now reads
\begin{align*}
	\|\eta_u^{n+1}\|_\Amat^2 - \|\eta_u^n&\|_\Amat^2 +
	\|\eta_p^{n+1}\|_\Cmat^2 - \|\eta_p^n\|_\Cmat^2 \\
	&\qquad\le
	\tau^3 \left[
	2\omega\, C_\text{Prop}^2\, e^{2T} 
	+ \tfrac14\, (1+C_\text{em}) \big( \omega\, \|\ddot{u}\|^2_{L^\infty(\Amat)} + \|\ddot{p}\|_{L^\infty(\Cmat)}^2 \big)
	\right].
\end{align*}
Finally, summing over $n$ yields 
\[
	\|\eta_u^n\|_\Amat^2 + \|\eta_p^n\|_\Cmat^2 
	\le 
	\tau^2\, T\, \left[
	2\omega\, C_\text{Prop}^2\, e^{2T} 
	+ \tfrac14\, (1+C_\text{em}) \big( \omega\, \|\ddot{u}\|^2_{L^\infty(\Amat)} + \|\ddot{p}\|_{L^\infty(\Cmat)}^2 \big)
	\right],
\]
which completes the proof.
\end{proof}
\begin{remark}
The iteration bound~\eqref{eq:iterationBound} can be solved for $K$, leading to the requirement
\[ 
	K 
	> K(\omega)
	\coloneqq 1 + \frac{\log(\omega)}{\log(\omega + 2) - \log(\omega)}. 
\]
Since this is well-defined for all positive~$\omega$, we see that stability and first-order convergence can be achieved for all choices of model parameters with an appropriately chosen number of inner iteration steps. A corresponding list of~$K$ and $\omega$ values is given in Table~\ref{tab:iterationBound}, see also Figure~\ref{fig:sharpnessComparison} below.
\end{remark}
\begin{table}
	\caption{Relation of the coupling parameter~$\omega$ and needed number of inner iterations~$K$ according to the bound~\eqref{eq:iterationBound}. }
	\begin{center}
		\begin{tabular}{c|c c c c c c c c c c}
			$K$ & 1 & 2 & 3 & 4 & 5 & 6 & 7 & 8 & 9 & 10 \\ 
			\hline
			$\omega <$ & 1.00 & 2.00 & 2.87 & 3.67 & 4.43 & 5.15 & 5.84 & 6.51 & 7.16 & 7.80
		\end{tabular}
	\end{center}
	\label{tab:iterationBound}
\end{table}
\begin{remark}
Following the same proof without Assumption~\ref{ass:initialStaticPressure}, one can still recover a convergence rate of order $\tau^{3/4}$. In numerical experiments, however, such a reduction of the order cannot be observed. 
%
%
%
\end{remark}
%
%
\section{Numerical Examples}\label{sec:numerics}
This section is devoted to a numerical investigation of the presented time stepping scheme. In a first experiment, we consider a toy problem in order to study the proven lower bound on the number of inner iterations. The second experiment then considers a real-world example, namely the simulation of brain tissue. Finally, we present a runtime comparison proving the competitiveness of the newly introduced scheme. 
%
%
\subsection{Sharpness of the iteration bound}\label{sec:numerics:sharpness}
We numerically analyze the sharpness of the iteration bound \eqref{eq:iterationBound}, which guarantees linear convergence. For this, we apply our scheme to the model problem from~\cite{AltMU22}, i.e., to~\eqref{eq:poro:weak} with bilinear forms 
%
\begin{align*}
	a(u, v) = v^T\Amat u, \qquad
	b(p, q) = q^TBp, \qquad 
	c(p, q) = q^T\Cmat p, \qquad
	d(v, p) = \sqrt{\tilde \omega}\, p^T\Dmat v
\end{align*}
with 
\begin{align*}
	\Amat = \tfrac{1}{2 - \sqrt{2}} 
	\begin{bmatrix} 2 & -1 & 0 \\ -1 & 2 & -1 \\ 0 & -1 & 2 \end{bmatrix}, \qquad
	\Bmat = 1, \qquad
	\Cmat = 1, \qquad
	\Dmat = \tfrac{1}{3}\, \big[\ 2\ \ 1\ \ 2\ \big]. 
\end{align*}
%
Note that $\Amat$ and $\Cmat$ are chosen such that their smallest eigenvalues equal~$1$, respectively, and that~$\Dmat$ is chosen such that the operator norm of $d$ is~$\sqrt{\tilde \omega}$. As a result, the coupling parameter introduced in Section~\ref{sec:prelim:omega} (as we are not considering poroelasticity here) is indeed given by~$\tilde \omega$. For the right-hand sides, we set
\[ 
	f(t) = \big[\ 1\ \  1\ \  1\ \big]^T, \qquad 
	g(t) = \sin(t).
\]
Further, we consider consistent initial data with $p(0)=1$. Note that this means that Assumption~\ref{ass:initialStaticPressure} is not satisfied. 
For the simulation, we use $T=1$ and a fixed time step size~$\tau = 1/300$. 
We run the scheme with $K$ ranging from $1$ to $5$ and compare the approximations with the implicit Euler solution acting as reference solution. The results are shown in Figure~\ref{fig:modelProblem}. Therein, it can be clearly observed that the scheme turns unstable if the coupling parameter~$\tilde \omega$ gets too large. 
\begin{figure}
%
%
%
\begin{tikzpicture}

\begin{axis}[%
width=0.7\textwidth,
height=0.3\textwidth,
scale only axis,
unbounded coords=jump,
xmin=0.8,
xmax=11,
xlabel style={font=\color{white!15!black}},
xlabel={coupling parameter~$\tilde \omega$},
ymode=log,
ymin=1e-06,
ymax=5,
yminorticks=true,
ylabel style={font=\color{white!15!black}},
ylabel={relative error},
axis background/.style={fill=white},
legend columns = 5,
legend style={at={(1.02,1.02)}, anchor=south east, legend cell align=left, align=left, draw=white!15!black}
]

\addplot [color=mycolor1, very thick]
  table[row sep=crcr]{%
0.8	0.000170839891311137\\
1.02857142857143	0.000170708410441701\\
1.25714285714286	76471.2042347891\\
1.48571428571429	4.06315127580839e+26\\
1.71428571428571	1.64364613045475e+45\\
};
\addlegendentry{$K=1$\quad}

\addplot [color=mycolor2, very thick, dashed]
  table[row sep=crcr]{%
0.8	6.46403646388838e-05\\
1.02857142857143	9.31630554121448e-05\\
1.25714285714286	0.000128106363184143\\
1.48571428571429	0.000172342811429901\\
1.71428571428571	0.000231145515297873\\
1.94285714285714	0.000314959837455314\\
2.17142857142857	0.0004473616155515\\
2.4	0.000694654850680127\\
2.62857142857143	0.00134253111096417\\
2.85714285714286	0.00852925950970248\\
3.08571428571429	576729808.479717\\
3.31428571428571	2.04544205981667e+21\\
};
\addlegendentry{$K=2$\quad}

\addplot [color=mycolor3, very thick, dotted]
  table[row sep=crcr]{%
0.8	1.07272136760963e-05\\
1.02857142857143	1.65524416002697e-05\\
1.25714285714286	2.27162536952146e-05\\
1.48571428571429	2.88282169292283e-05\\
1.71428571428571	3.46082057598755e-05\\
1.94285714285714	3.98799654686038e-05\\
2.17142857142857	4.4551370264254e-05\\
2.4	4.85918699661295e-05\\
2.62857142857143	5.20123580130314e-05\\
2.85714285714286	5.48492738003801e-05\\
3.08571428571429	5.71529971050258e-05\\
3.31428571428571	5.89798831653135e-05\\
3.54285714285714	6.0387092191421e-05\\
3.77142857142857	6.14294479906112e-05\\
4	6.21576959731684e-05\\
4.22857142857143	6.26176844585822e-05\\
4.45714285714286	6.28501399321648e-05\\
4.68571428571429	6.28907935815553e-05\\
4.91428571428571	6.01921649852307e-05\\
5.14285714285714	16994.9472791392\\
5.37142857142857	10311400650619.9\\
5.6	2.17887037082638e+21\\
};
\addlegendentry{$K=3$\quad}

\addplot [color=mycolor4, very thick, dashdotdotted]
  table[row sep=crcr]{%
0.8	2.17013603852014e-06\\
1.02857142857143	4.08701093251274e-06\\
1.25714285714286	6.59905146675759e-06\\
1.48571428571429	9.64684182798573e-06\\
1.71428571428571	1.31696588719523e-05\\
1.94285714285714	1.71175063293752e-05\\
2.17142857142857	2.14563036049153e-05\\
2.4	2.61694611266267e-05\\
2.62857142857143	3.12579724836966e-05\\
2.85714285714286	3.67402604467667e-05\\
3.08571428571429	4.26524766217181e-05\\
3.31428571428571	4.90496826601899e-05\\
3.54285714285714	5.60082853549801e-05\\
3.77142857142857	6.36301477424023e-05\\
4	7.20490070568775e-05\\
4.22857142857143	8.1440198532336e-05\\
4.45714285714286	9.20353063238263e-05\\
4.68571428571429	0.000104144493358703\\
4.91428571428571	0.000118191263183086\\
5.14285714285714	0.0001347682304563\\
5.37142857142857	0.000154730084705463\\
5.6	0.00017935604147848\\
5.82857142857143	0.000210650776102441\\
6.05714285714286	0.000251944177723023\\
6.28571428571428	0.000309204697037027\\
6.51428571428571	0.000394302325158099\\
6.74285714285714	0.000534711932188834\\
6.97142857142857	0.000811798913728293\\
7.2	0.00162159016451268\\
7.42857142857143	0.0205646287443043\\
7.65714285714286	2888.80776344298\\
7.88571428571429	2164637838.54074\\
8.11428571428571	1.30497234339298e+15\\
};

\addlegendentry{$K=4$\quad}

\addplot [color=mycolor5, very thick, dashdotted]
  table[row sep=crcr]{%
0.8	4.23161411677418e-07\\
1.02857142857143	9.4020709068871e-07\\
1.25714285714286	1.70544993450225e-06\\
1.48571428571429	2.70964382610866e-06\\
1.71428571428571	3.92401739528441e-06\\
1.94285714285714	5.30920886422562e-06\\
2.17142857142857	6.82208372951827e-06\\
2.4	8.42030113286029e-06\\
2.62857142857143	1.00650477776844e-05\\
2.85714285714286	1.17224476815221e-05\\
3.08571428571429	1.33640800798121e-05\\
3.31428571428571	1.4966941168873e-05\\
3.54285714285714	1.65130725521525e-05\\
3.77142857142857	1.79890219536808e-05\\
4	1.93852361739278e-05\\
4.22857142857143	2.06954465729146e-05\\
4.45714285714286	2.19160936702296e-05\\
4.68571428571429	2.30458024041143e-05\\
4.91428571428571	2.40849223042476e-05\\
5.14285714285714	2.50351330435092e-05\\
5.37142857142857	2.58991111491448e-05\\
5.6	2.66802547339368e-05\\
5.82857142857143	2.73824575643801e-05\\
6.05714285714286	2.80099292939814e-05\\
6.28571428571428	2.85670503199455e-05\\
6.51428571428571	2.90582607872015e-05\\
6.74285714285714	2.9487975477738e-05\\
6.97142857142857	2.98605171233758e-05\\
7.2	3.01800717553772e-05\\
7.42857142857143	3.04506548418699e-05\\
7.65714285714286	3.06760891609693e-05\\
7.88571428571429	3.08599917010124e-05\\
8.11428571428571	3.10057675349805e-05\\
8.34285714285714	3.11166083099198e-05\\
8.57142857142857	3.11954945438374e-05\\
8.8	3.12452014264334e-05\\
9.02857142857143	3.12683064223288e-05\\
9.25714285714286	3.12671969178582e-05\\
9.48571428571429	3.12440814899212e-05\\
9.71428571428571	3.12009987556023e-05\\
9.94285714285714	3.11398280758505e-05\\
10.1714285714286	3.10623000456665e-05\\
10.4	3.096755949202e-05\\
10.6285714285714	0.000163244950239044\\
10.8571428571429	5.82962728320501\\
11.0857142857143	141906.227910051\\
11.3142857142857	2631937149.93072\\
11.5428571428571	37682701947620.7\\
11.7714285714286	4.2159719770477e+17\\
};
\addlegendentry{$K=5$}

\end{axis}
\end{tikzpicture}%
	\caption{Relative error (measured in the Euclidean norm) of the proposed scheme applied to the model problem of Section~\ref{sec:numerics:sharpness} in comparison to the implicit Euler solution at final time $T = 1$. }
	\label{fig:modelProblem}
\end{figure}
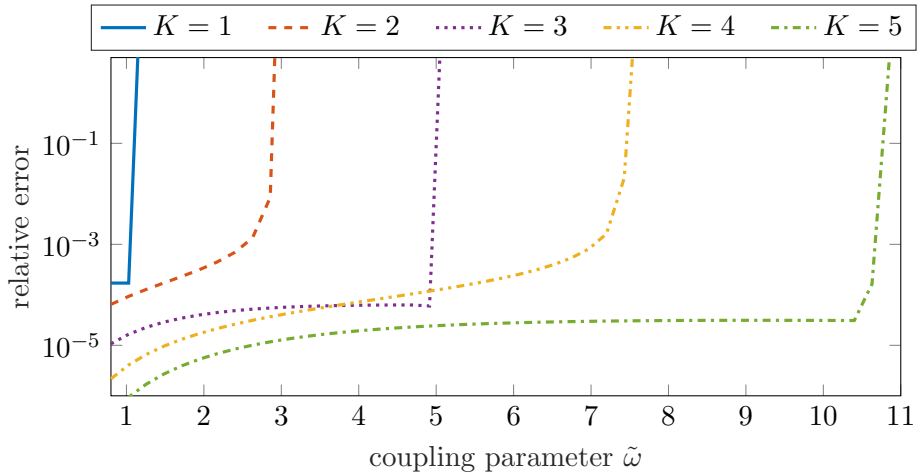

This experiment also shows that the proven bound~\eqref{eq:iterationBound} is quite pessimistic in the sense that the proposed scheme may be used for coupling parameters exceeding this bound. This fact is illustrated in Figure~\ref{fig:sharpnessComparison}. Therein, we plot the function $K(\tilde \omega)$ given by the theoretical bound~\eqref{eq:iterationBound} and the largest value of~$\tilde \omega$ for which the scheme is still convergent for the above model problem. 
We make this more precise for the poroelastic example of shale with coupling parameter~$\omega=4.02$ (see Table~\ref{tab:constants}). Theorem~\ref{thm:general-convergence} guarantees first-order convergence for $K\ge5$ but the model problem indicates that three inner iteration steps are already sufficient. 
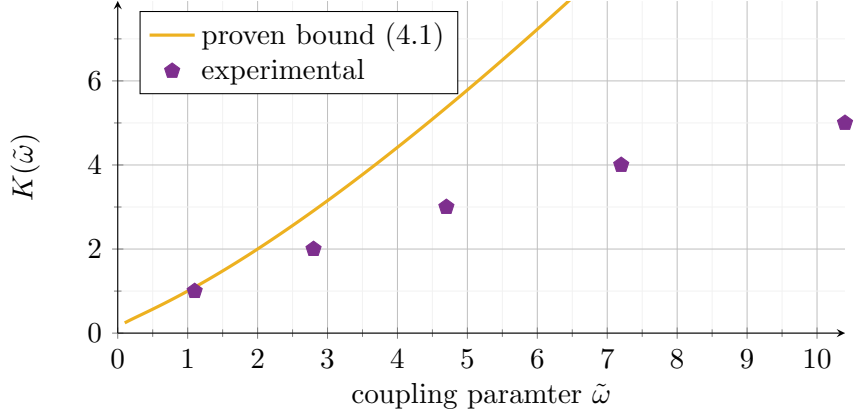
\begin{figure}
	\begin{tikzpicture}
		\begin{axis}[
			width=0.75\textwidth,
			height=0.4\textwidth,
			axis lines = left,
			xlabel = coupling paramter~$\tilde \omega$,
			ylabel = \(K(\tilde \omega)\),
			xmin = 0, xmax=10.4,
			ymin = 0, ymax=7.9,
			grid = both,
		    grid style={line width=.1pt, draw=gray!10},
		    major grid style={line width=.2pt,draw=gray!50},
			minor x tick num=1,
			minor y tick num=1,
			legend pos = north west,
			legend cell align = {left},
		]
			\addplot [
				domain = 0.1:8,
				samples = 100,
				variable = \t,
				color = mycolor4,
				very thick,
			]
			({t}, {1 + ln(t)/(ln((t + 2)/t))});
			\addplot [
			mark = pentagon*,
			mark size = 3,
			color = mycolor3,
			only marks,
			] coordinates
			{
				(1.1, 1)
				(2.8, 2)
				(4.7, 3)
				(7.2, 4)
				(10.4, 5)
			};
			\legend{proven bound~\eqref{eq:iterationBound}, experimental}
		\end{axis}
	\end{tikzpicture}
	\caption{Comparison of the proven iteration bound~\eqref{eq:iterationBound} and experimental iteration bounds for the model problem of Section~\ref{sec:numerics:sharpness}.  }
	\label{fig:sharpnessComparison}
\end{figure}
%
%
\subsection{Application to brain tissue}\label{sec:numerics:brain}
In this second experiment, we apply the proposed scheme to a more complex example taken from~\cite{JuCLT20}. Therein, the poroelastic model is used to simulate the deformation of brain tissue as a result of brain edema. We would like to confirm our theoretical results, which predict our scheme to be efficient for moderate coupling parameters in the range~$\omega \in [1,10]$. As such, we consider the model parameters as in~\cite[Tab.~10]{JuCLT20} but replace their choice of the Biot modulus~$M$ -- which would lead to $\omega \approx 200$ -- by a lower value. Indeed, a wide range of values for the Biot modulus is used within medical literature, cf.~\cite{EliRT21,JuCLT20,PieRV22} or Table~\ref{tab:constants} where an even smaller value of $M$ is used. To summarize, our model parameters are
\begin{align*}
	\lambda &= 7.8\times10^3 \newton/\meter^2, &
	\mu &= 3.3\times10^3 \newton/\meter^2, &
	\alpha &= 1, \\
	\kappa &= 1.3\times10^{-15} \meter^2, &
	\nu &= 8.9\times10^{-4} \newton\second/\meter^2, &
	M &= 2.2\times10^{4} \newton/\meter^2,
\end{align*}
leading to~$\omega \approx 2.8$. Hence, the proven bound~\eqref{eq:iterationBound} claims~$K = 3$ to be sufficient, which we will verify in the following.

As in~\cite{JuCLT20}, we construct a two-dimensional triangular mesh with $11221$ elements from brain scans obtained from \cite{JohB01}. The model specifies mixed boundary conditions, namely
\begin{subequations} 
\label{eq:brain:bc}
\begin{align}
	u &= 0 && \text{on } \Gamma_1, \\
	\left(\tfrac{\kappa}{\nu}\, \nabla p\right) \cdot n 
	&= c_{\operatorname{SAS}} (p_{\operatorname{SAS}} - p) && \text{on } \Gamma_1, \\
	(\sigma(u) - \alpha p) \cdot n 
	&= -p \cdot n && \text{on } \Gamma_2, \label{eq:brain:neumann-bc} \\
	p 
	&= 1100 \newton/\meter^2 && \text{on } \Gamma_2. \label{eq:brain:dirichlet-bc}
\end{align}
\end{subequations}
Here, $\Gamma_1$ is the outer border of the mesh representing the brain tissue wall and $\Gamma_2$ is the inner border representing the ventricular wall. Moreover, $n$ denotes the outer normal vector, $c_{\operatorname{SAS}} = 5.0\times10^{-10} \meter^3/(\newton\second)$ the conductance, and~$p_{\operatorname{SAS}} = 1070 \newton/\meter^2$ the pressure outside the brain tissue wall. Note that the conditions~\eqref{eq:brain:bc} do not match with the homogeneous Dirichlet boundary conditions we have assumed in the above analysis. Nevertheless, the resulting matrices can be constructed in such a way that their symmetric structure is preserved. As such, all assumptions on the matrices used in our analysis are still fulfilled. 

We first compute the neutral state $(u^0, p^0)$ of displacement and pressure as the solution to the linear system consisting of the boundary conditions~\eqref{eq:brain:bc} and 
\begin{subequations} 
\label{eq:brain:static}
\begin{align}
	-\nabla \cdot \sigma(u^0) + \alpha \nabla p^0 
	&= 0 \qquad\quad\text{on } \Omega, \\
	\nabla \cdot \left(\tfrac{\kappa}{\nu}\, \nabla p^0\right) 
	&= 0 \qquad\quad\text{on } \Omega.
\end{align}
\end{subequations}
This then serves as initial data. Subsequently, a non-zero source term $g = 1.5\times10^{-4}\second^{-1}$ is introduced in the damaged part of the brain. The corresponding differential equation is then solved using the scheme~\eqref{eq:scheme} for~$T = 4.2 \hour$, $N = 100$, and~$K = 2$. The initial and final distributions for the pressure and the displacement are shown in Figure~\ref{fig:brain}(A). In the neutral state, there is almost no displacement and~$p$ stays between $1070 \newton/\meter^2$ and $1100 \newton/\meter^2$ as dictated by the boundary conditions. For~$t = T$, however, we see a maximum pressure of around $2000 \newton/\meter^2$ at the injured region and a displacement of up to $0.2 \millimeter$. 

We also use this example to compare convergence properties of the proposed and the fixed-stress scheme for fixed numbers of inner iterations. For this, we run both schemes with varying time step sizes~$\tau$, while the implicit Euler solution with $\tau = 1.5 \second$ is used as a reference. As the solution converges to an equilibrium over time, we set~$T = 10 \minute$ and compare the solutions at this point in time. Note that the computational costs for both schemes are comparable for the same number of inner iterations. The computational results can be seen in Figure~\ref{fig:brain}(B). 
As expected, we observe that the proposed scheme is not stable for~$K = 1$. For two and three inner iterations, however, the scheme converges and produces almost the same results. The fixed-stress scheme, on the other hand, is stable for any number of inner iterations, but does not seem to converge towards the reference solution for $\tau \to 0$. Instead, the error reaches a certain plateau, which gets smaller for increasing number of inner iteration steps. This may be caused by the relaxation in the final step, cf.~the discussion in Remark~\ref{rem:final-relaxation}.  
\begin{figure}
	\begin{subfigure}[b]{0.45\textwidth}
		\includegraphics[width=\textwidth]{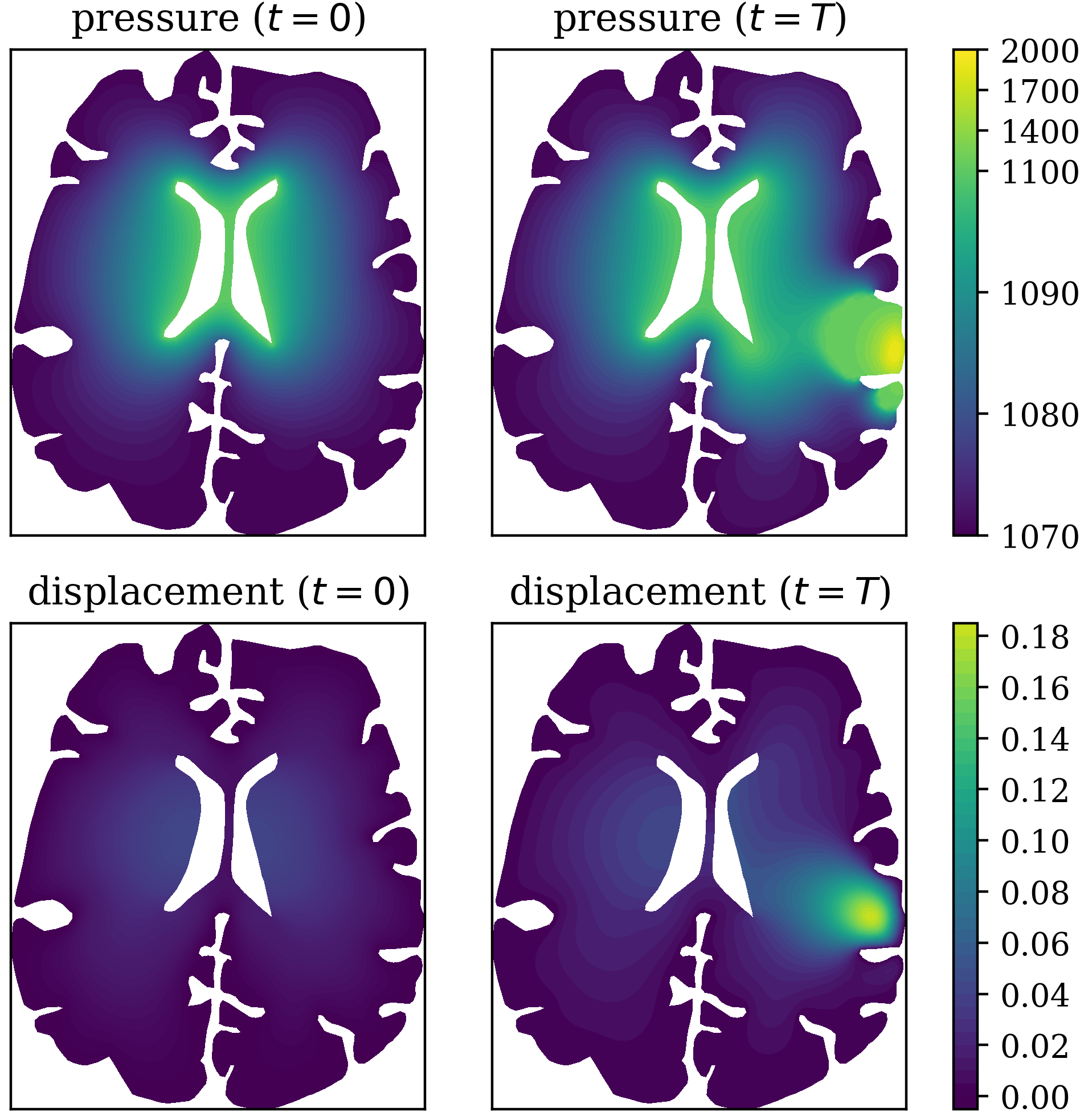}
		\caption{Initial and final state for the pressure (in $\newton/\meter^2$) and the displacement (in $\millimeter$). A non-zero source term is applied in the ``damaged'' subdomain on the right of the brain.
		}
	\label{fig:brain:final}
	\end{subfigure}
	\qquad
	\begin{subfigure}[b]{0.49\textwidth}
		\begin{tikzpicture}
		\begin{axis}[%
		width=0.8\textwidth,
		height=0.6\textwidth,
		at={(0\textwidth,0\textwidth)},
		scale only axis,
		unbounded coords=jump,
		xlabel style={font=\color{white!15!black}},
		xlabel={$\tau$ (in $\second$)},
		xmode=log,
		ymode=log,
		yminorticks=true,
		ylabel style={font=\color{white!15!black}},
		ylabel={relative error},
		axis background/.style={fill=white},
		legend columns = 4,
		legend style={
			at={(1.0,1.15)},
			anchor=east,
			legend cell align=left,
			align=left,
			draw=white!15!black,
			font=\Small
		},
		xmin = 2.7,
		xmax = 62,
		ymin = 0.00002,
		ymax = 0.6,
		mark repeat=3,
		mark size=3.5,
		]
		\addlegendimage{empty legend}
		\addlegendentry{scheme~\eqref{eq:scheme}}

		\addplot [color = mycolor4, mark=triangle*, very thick] table [col sep=semicolon, x expr = 600/(1+\thisrowno{0}), y index = {1}] {experiments/output/error_tau.csv};
		\addlegendentry{$K = 1$}
		\addplot [color = mycolor4, mark=x, very thick] table [col sep=semicolon, x expr = 600/(1+\thisrowno{0}), y index = {2}] {experiments/output/error_tau.csv};
		\addlegendentry{$K = 2$}
		\addplot [color = mycolor4, mark=o, very thick] table [col sep=semicolon, x expr = 600/(1+\thisrowno{0}), y index = {3}] {experiments/output/error_tau.csv};
		\addlegendentry{$K = 3$}

		\addlegendimage{empty legend}
		\addlegendentry{fixed-stress}

		\addplot [color = mycolor3, dashed, mark=triangle*, thick, every mark/.append style={solid}] table [col sep=semicolon, x expr = 600/(1+\thisrowno{0}), y index = {4}] {experiments/output/error_tau.csv};
		\addlegendentry{$K = 1$}
		\addplot [color = mycolor3, dashed, mark=x, thick, every mark/.append style={solid}] table [col sep=semicolon, x expr = 600/(1+\thisrowno{0}), y index = {5}] {experiments/output/error_tau.csv};
		\addlegendentry{$K = 2$}
		\addplot [color = mycolor3, dashed, mark=o, thick, every mark/.append style={solid}] table [col sep=semicolon, x expr = 600/(1+\thisrowno{0}), y index = {6}] {experiments/output/error_tau.csv};
		\addlegendentry{$K = 3$}
		\end{axis}
		\end{tikzpicture}%
	\subcaption{Relative error of the proposed and the fixed-stress scheme for~$T = 10 \minute$ and varying~$\tau$. The reference solution is computed by the implicit Euler scheme with~$\tau = 1.5 \second$. }
	\label{fig:brain:error-tau}
	\end{subfigure}
	\caption{Numerical results for the brain model of Section~\ref{sec:numerics:brain}.}
	\label{fig:brain}
\end{figure}
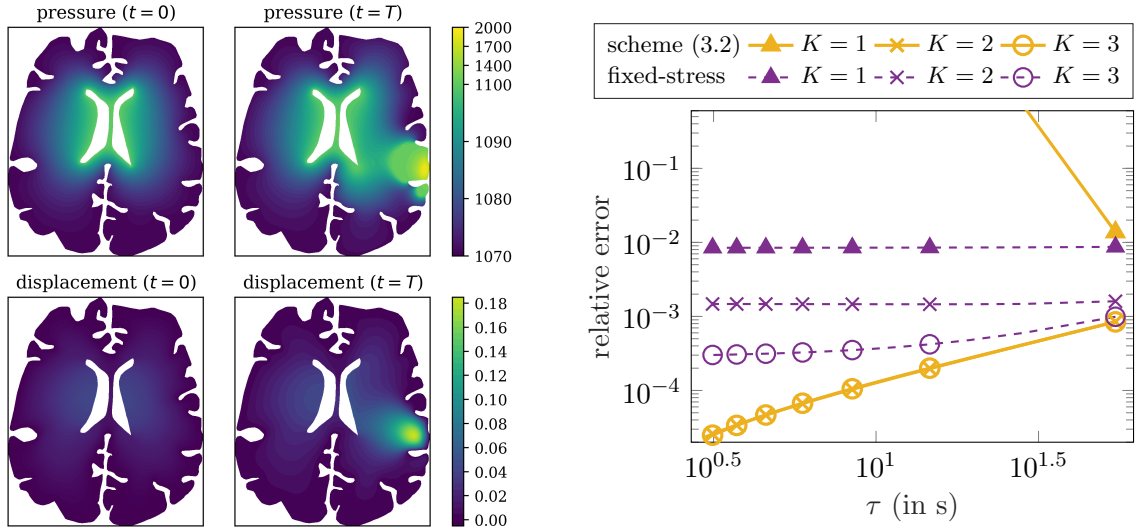
\subsection{Implementation details and runtime comparison}
For an efficient overall implementation, one also needs to discuss preconditioners and the linear solvers used for the resulting systems in each time step. 
\subsubsection*{preconditioners}
As already mentioned in the introduction, decoupled approaches have the advantage that we need to solve two smaller subsystems for which well-known preconditioners exist~\cite{LeeMW17}. 
Hence, we first consider preconditioners for the linear systems with left-hand side matrices $\Amat$ and $\Ctau$ that come up in semi-explicit methods. Here classical preconditioners for linear elasticity and Darcy flows can be used, respectively. In our experiments, we found that $\Amat$ can be effectively preconditioned using an {\em algebraic multigrid (AMG) preconditioner}~\cite{YanH02}, while we use a {\em Jacobi preconditioner} for $\Ctau$ where the mass matrix is dominating. We would like to emphasize that we do not consider the nearly incompressible case here. For this, one has to take special care of the spatial discretization, cf.~\cite{LeePMR19}.  

For a fully implicit discretization, one needs a block preconditioner. Combining the previous two preconditioners, however, does not yield a good preconditioner of the full system. Rather common is a construction based on the Schur complement $\Smat = \Dmat^T \Amat \Dmat + \Ctau^{-1}$, see~\cite{BenGL05}. Using AMG, for example, where $\Ctau^{-1}$ can be effectively approximated by the inverse diagonal of $\Ctau$, one obtains the preconditioner
\[ 
	\begin{bmatrix} \Amat^{-1} & 0 \\ 0 & \Smat^{-1} \end{bmatrix}, 
\]
where both inverses are approximated. As already mentioned, there are more complex preconditioners for the implicit approach which promise effectiveness independent of certain parameters~\cite{LeeMW17, HonK18}. In our examples, however, the relation between the parameters is rather moderate and we found them to be less efficient than a solution based on the Schur complement.
\subsubsection*{linear solvers}
In the decoupled case, the matrices $\Amat$ and $\Ctau$ are symmetric and positive definite. Hence, we apply the {\em preconditioned conjugate gradient method} for both subsystems. For an implicit time discretization, on the other hand, the conjugate gradient method does not necessarily converge, since the full system is not positive definite. We can, however, still exploit the symmetric structure of the system by using {\em MinRes}. 

For the proposed scheme, it can be seen from the proof of Theorem~\ref{thm:general-convergence} that the main importance of the first $K-1$ steps is only to improve the consistency of the variables $p$ and $u$ with regards to the first equation. Hence, no high accuracy is required at this point and one can perform these steps in an inexact manner. Indeed, it can be observed experimentally that already a few iterations of the linear solver (or even only preconditioner steps) are sufficient. 
%
%
\subsubsection*{runtime comparison}
We conclude this section with a runtime comparison of the fixed-stress, the implicit Euler, and the newly proposed scheme~\eqref{eq:scheme}. For this, we use the example from Section~\ref{sec:numerics:brain} and vary the Biot modulus~$M$, leading to problems with $\omega = 2.8$ (the original setting) as well as $\omega=10$ and $\omega=100$. All calculations were run on two Xeon E5-2630 processors (together having 16 cores of 2.4 GHz). 
The corresponding results are shown in Figure~\ref{fig:runtimeComparison}. Therein, one can clearly see the linear dependence between runtime and error as long as the considered schemes converge. As already observed in the previous subsection, fixed-stress only shows the linear rate as long as the number of inner iterations~$K$ is sufficiently large. Moreover, the fully implicit scheme has much larger runtimes than the novel iterative approach. 
This even holds true for larger $\omega$, see Figure \ref{fig:runtimeComparison2}. For $\omega=100$, however, the advantages of the proposed scheme compared to the fixed-stress scheme get lost.
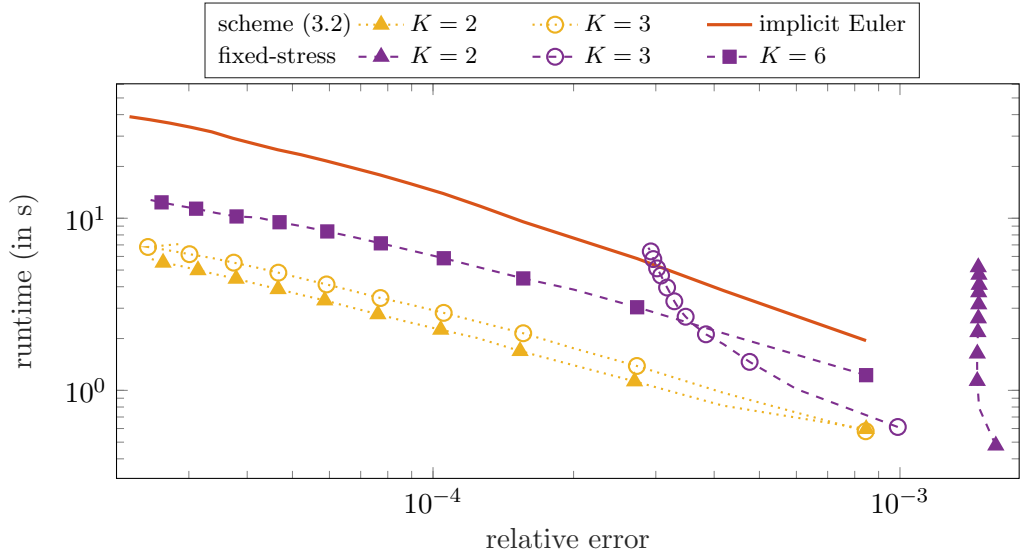
\begin{figure}
\begin{tikzpicture}
		\begin{axis}[%
		width=0.8\textwidth,
		height=0.35\textwidth,
		at={(0\textwidth,0\textwidth)},
		scale only axis,
		unbounded coords=jump,
		xlabel style={font=\color{white!15!black}},
		xlabel={relative error},
		xmode=log,
		ymode=log,
		ylabel style={font=\color{white!15!black}},
		ylabel={runtime (in $\second$)},
		axis background/.style={fill=white},
		legend columns = 4,
		legend style={
			at={(0.89,1.11)},
			anchor=east,
			legend cell align=left,
			align=left,
			draw=white!15!black,
			font=\Small
		},
		xmin = 0.000021,
		xmax = 0.0018,
		mark repeat=2,
		mark size=3.0,
		]
		\addlegendimage{empty legend}
		\addlegendentry{scheme~\eqref{eq:scheme}}

		\addplot [color = mycolor4, dotted, mark=triangle*, thick, every mark/.append style={solid}] table [col sep=semicolon, x index = {1}, y index = {7}] {experiments/output/runtime_comparison_omega_2.csv};
		\addlegendentry{$K = 2$\qquad}
		\addplot [color = mycolor4, dotted, mark=o, thick, every mark/.append style={solid}] table [col sep=semicolon, x index = {2}, y index = {8}] {experiments/output/runtime_comparison_omega_2.csv};
		\addlegendentry{$K = 3$\qquad}
		
		\addlegendentry{implicit Euler}
		
		\addplot [color = mycolor2, very thick, every mark/.append style={solid}] table [col sep=semicolon, x index = {6}, y index = {12}] {experiments/output/runtime_comparison_omega_2.csv};
		
		\addlegendimage{empty legend}		
		\addlegendentry{fixed-stress}

		\addplot [color = mycolor3, dashed, mark=triangle*, thick, every mark/.append style={solid}] table [col sep=semicolon, x index = {3}, y index = {9}] {experiments/output/runtime_comparison_omega_2.csv};
		\addlegendentry{$K = 2$\quad}
		\addplot [color = mycolor3, dashed, mark=o, thick, every mark/.append style={solid}] table [col sep=semicolon, x index = {4}, y index = {10}] {experiments/output/runtime_comparison_omega_2.csv};
		\addlegendentry{$K = 3$\quad}
		\addplot [color = mycolor3, dashed, mark=square*, mark size=2.3, thick, every mark/.append style={solid}] table [col sep=semicolon, x index = {5}, y index = {11}] {experiments/output/runtime_comparison_omega_2.csv};
		\addlegendentry{$K = 6$}

		\addlegendimage{empty legend}
		\end{axis}
		\end{tikzpicture}
		\caption{Comparison of runtimes of the fixed-stress, the implicit Euler and the novel iterative scheme~\eqref{eq:scheme} for different numbers of inner iterations and $\omega = 2.8$.  }
	\label{fig:runtimeComparison}
\end{figure}

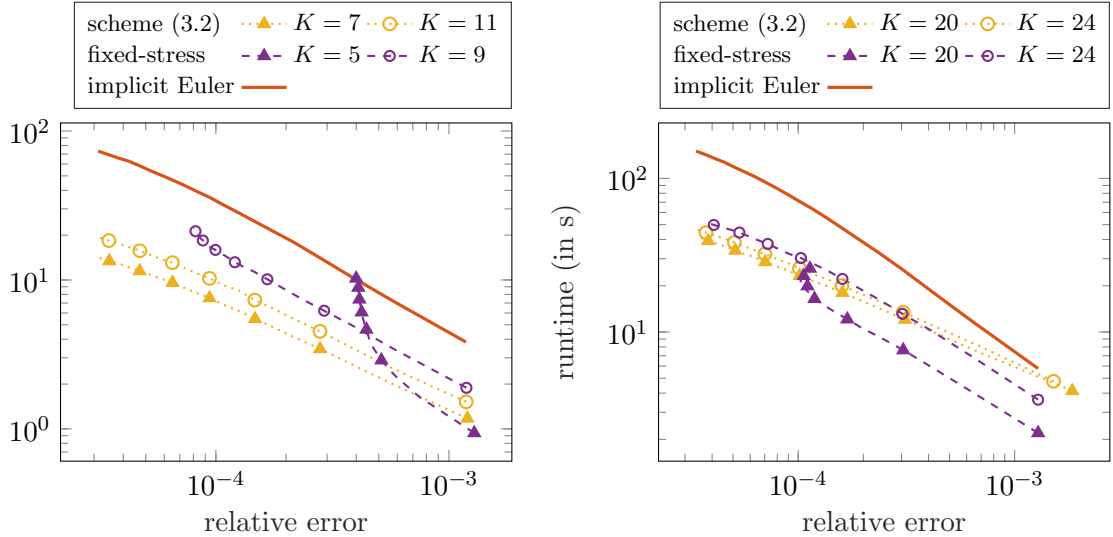
\begin{figure}
\begin{tikzpicture}
		\begin{axis}[%
		width=0.4\textwidth,
		height=0.3\textwidth,
		at={(0\textwidth,0\textwidth)},
		scale only axis,
		unbounded coords=jump,
		xlabel style={font=\color{white!15!black}},
		xlabel={relative error},
		xmode=log,
		ymode=log,
		ylabel style={font=\color{white!15!black}},
		axis background/.style={fill=white},
		legend columns = 3,
		legend style={
			at={(1,1.2)},
			anchor=east,
			legend cell align=left,
			align=left,
			draw=white!15!black,
			font=\Small
		},
		mark repeat=3,
		mark size=2.4,
		]
		\addlegendimage{empty legend}
		\addlegendentry{scheme~\eqref{eq:scheme}}

		\addplot [color = mycolor4, dotted, mark=triangle*, thick, thick, every mark/.append style={solid}] table [col sep=semicolon, x index = {1}, y index = {7}] {experiments/output/runtime_comparison_omega_10.csv};
		\addlegendentry{$K = 7$}
		\addplot [color = mycolor4, dotted, mark=o, thick, thick, every mark/.append style={solid}] table [col sep=semicolon, x index = {2}, y index = {8}] {experiments/output/runtime_comparison_omega_10.csv};
		\addlegendentry{$K = 11$}

		\addlegendimage{empty legend}
		\addlegendentry{fixed-stress}

		\addplot [color = mycolor3, dashed, mark=triangle*, thick, every mark/.append style={solid}] table [col sep=semicolon, x index = {3}, y index = {9}] {experiments/output/runtime_comparison_omega_10.csv};
		\addlegendentry{$K = 5$}
		\addplot [color = mycolor3, dashed, mark=o, thick, every mark/.append style={solid}, mark size=1.9] table [col sep=semicolon, x index = {5}, y index = {11}] {experiments/output/runtime_comparison_omega_10.csv};
		\addlegendentry{$K = 9$}

		\addlegendimage{empty legend}
		\addlegendentry{implicit Euler}

		\addplot [color = mycolor2, very thick] table [col sep=semicolon, x index = {6}, y index = {12}] {experiments/output/runtime_comparison_omega_10.csv};
		\addlegendentry{}
		\addlegendimage{empty legend}
		\addlegendentry{}
		\end{axis}
%
		\begin{axis}[%
			width=0.4\textwidth,
			height=0.3\textwidth,
			at={(0.53\textwidth,0\textwidth)},
			scale only axis,
			unbounded coords=jump,
			xlabel style={font=\color{white!15!black}},
			xlabel={relative error},
			xmode=log,
			ymode=log,
			ylabel style={font=\color{white!15!black}},
			ylabel={runtime (in $\second$)},
			axis background/.style={fill=white},
			legend columns = 3,
			legend style={
				at={(1.0,1.2)},
				anchor=east,
				legend cell align=left,
				align=left,
				draw=white!15!black,
				font=\Small
			},
			mark repeat=3,
			mark size=2.4,
			]
			\addlegendimage{empty legend}
			\addlegendentry{scheme~\eqref{eq:scheme}}

			\addplot [color = mycolor4, mark=triangle*, dotted, thick, thick, every mark/.append style={solid}] table [col sep=semicolon, x index = {1}, y index = {7}] {experiments/output/runtime_comparison_omega_100.csv};
			\addlegendentry{$K = 20$}
			\addplot [color = mycolor4, mark=o, thick, dotted, thick, every mark/.append style={solid}] table [col sep=semicolon, x index = {2}, y index = {8}] {experiments/output/runtime_comparison_omega_100.csv};
			\addlegendentry{$K = 24$}

			\addlegendimage{empty legend}
			\addlegendentry{fixed-stress}
			
			\addplot [color = mycolor3, dashed, mark=triangle*, thick, every mark/.append style={solid}] table [col sep=semicolon, x index = {3}, y index = {9}] {experiments/output/runtime_comparison_omega_100.csv};
			\addlegendentry{$K = 20$}
			\addplot [color = mycolor3, dashed, mark=o, thick, every mark/.append style={solid}, mark size=1.9] table [col sep=semicolon, x index = {5}, y index = {11}] {experiments/output/runtime_comparison_omega_100.csv};
			\addlegendentry{$K = 24$}

			\addlegendimage{empty legend}
			\addlegendentry{implicit Euler}

			\addplot [color = mycolor2, very thick] table [col sep=semicolon, x index = {6}, y index = {12}] {experiments/output/runtime_comparison_omega_100.csv};
			\addlegendentry{}
		\end{axis}
	\end{tikzpicture}
	\caption{Comparison of runtimes of the three schemes for different numbers of inner iterations and $\omega=10$ (left) and $\omega=100$ (right).  }
	\label{fig:runtimeComparison2}
\end{figure}
%
%
\section{Conclusions}
In this work, we have introduced a novel time stepping scheme for linear poroelasticity which decouples the elastic and the flow equation. For this, we combine ideas from classical iterative and non-iterative semi-explicit schemes. 
We have proven first-order convergence with an a priori bound on the number of inner iterations depending only on the coupling parameter~$\omega$. This allows the application to a larger class of materials (in contrast to the semi-explicit Euler scheme~\cite{AltMU21}) without the need of further stabilization parameters (in contrast to iterative schemes). Numerical experiments further show the competitiveness of the proposed scheme, especially for moderate coupling parameters~$\omega$ as they appear, e.g., in geomechanical applications.
%
%
\section*{Acknowledgments} 
Both authors acknowledge the support of the Deutsche Forschungsgemeinschaft (DFG, German Research Foundation) through the project 467107679. 
Moreover, parts of this work were carried out while the first author was affiliated with the Institute of Mathematics and the Centre for Advanced Analytics and Predictive Sciences (CAAPS) at the University of Augsburg.
%
%
\bibliographystyle{alpha}
\bibliography{references}
\end{document}